\documentclass[12pt]{amsart}

\numberwithin{equation}{section}

\usepackage{amsfonts,amssymb,amsmath,amsthm}

\usepackage{enumerate}
\usepackage{color}
\usepackage{latexsym}
\usepackage{algorithmic}
\usepackage{algorithm}
\usepackage[colorlinks]{hyperref}

\theoremstyle{plain}
\newtheorem{theorem}{Theorem}[section]
\newtheorem{corollary}[theorem]{Corollary}
\newtheorem{proposition}[theorem]{Proposition}
\newtheorem{lemma}[theorem]{Lemma}

\theoremstyle{definition}
\newtheorem{definition}[theorem]{Definition}

\theoremstyle{remark}
\newtheorem{remark}[theorem]{Remark}
\newtheorem{example}[theorem]{Example}
\newtheorem{conj}[theorem]{Conjecture}
\numberwithin{equation}{section}

\newcommand{\Proj}{\textnormal{Proj}\,}
\newcommand{\Supp}{\textnormal{Supp}\,}
\newcommand{\In}{\mathrm{in}}
\newcommand{\Gin}{\mathrm{gin}}
\newcommand{\IN}{\mathbf{in}}
\newcommand{\GIN}{\mathbf{gin}}
\newcommand{\supp}{\Delta\mathrm{Supp}}
\newcommand{\sat}{{\textnormal{sat}}}
\newcommand{\reg}{\textnormal{reg}}
\newcommand{\PP}{\mathbb{P}}
\newcommand{\GLin}{GL}
\newcommand{\GrassScheme}[2]{\mathbf{Gr}_{#1}^{#2}}
\newcommand{\GrassFunctor}[2]{\underline{\mathbf{Gr}}_{#1}^{#2}}
\newcommand{\vs}[1]{\mathtt{vs}(#1)}
\newcommand{\HilbScheme}[2]{\mathbf{Hilb}_{#1}^{#2}}
\newcommand{\HilbSchemeR}[3]{\mathbf{Hilb}_{#1}^{#2,[#3]}}
\renewcommand{\O}{\mathcal O}

\title
{Double-Generic initial ideal and Hilbert scheme}

\author[]{C. Bertone}
\address{Dip. di Matematica dell'Universit\`a di Torino, Torino, Italy}
\email{cristina.bertone@unito.it}

\author[]{F. Cioffi}
\address{Dip. di Matematica e Applicazioni dell'Universit\`a di Napoli, Napoli, Italy}
\email{cioffifr@unina.it}

\author[]{M. Roggero}
\address{Dip. di Matematica dell'Universit\`a di Torino, Torino, Italy}
\email{margherita.roggero@unito.it}
\thanks{The second author was partially supported by GNSAGA (INdAM, Italy). The second and third author were partially supported by the framework of PRIN 2010-11 \lq\lq Geometria delle variet\`a\ algebriche\rq\rq, cofinanced by MIUR}

\begin{document}

\maketitle

\begin{abstract} 
Following the approach in the book \lq\lq Commutative Algebra\rq\rq, by D. Eisenbud, where the author describes the generic initial ideal by means of a suitable total order on the terms of an exterior power, we introduce first the {\em generic initial extensor} of a subset  of a Grassmannian and then the {\em double-generic initial ideal} of a so-called \emph{\GLin-stable subset}  of a Hilbert scheme. We discuss the features of these new notions and introduce also a partial order which gives another useful description of them. 
The double-generic initial ideals turn out to be the appropriate points to understand some geometric pro\-per\-ties of a Hilbert scheme: they provide a necessary condition for a Borel ideal to correspond to a point of a given irreducible component, lower bounds for the number of irreducible components in a Hilbert scheme and the maximal Hilbert function in every irreducible component. Moreover, we prove that every isolated component having a smooth double-generic initial ideal is rational. As a byproduct, we prove that the Cohen-Macaulay locus of the Hilbert scheme parameterizing subschemes of codimension 2 is the union of open subsets isomorphic to affine spaces. This improves results by J. Fogarty (1968) and R. Treger (1989).

\end{abstract}


\section*{Introduction}

Let $\HilbScheme{p(t)}{n}$ denote the Hilbert scheme parameterizing all subschemes in a projective space $\mathbb P^n_K$ with Hilbert polynomial $p(t)$ over an infinite field $K$.
The Hilbert scheme was first introduced by Grothendieck \cite{Gro} in the $60$s and, although it has been intensively studied by several authors, its structure and features are still quite mysterious. Also the topological structure of a Hilbert scheme is not well-understood yet. For instance, in general it is not  known how many irreducible components a Hilbert scheme has and which of them are rational. In these topics, some few special cases have been treated, for example, by J. Fogarty \cite{Fo}, R. Piene and M. Schlessinger \cite{RaSc}, R. Treger \cite{TregerIII}, P. Lella and M. Roggero \cite{LR}.

The aim of this paper is to develop algebraic constructive methods in the context of the computation of initial and generic initial ideals, in order to study properties of the Hilbert scheme. Therefore, we often identify a point of $\HilbScheme{p(t)}{n}$ with any  ideal  in $S:=K[x_0, \dots, x_n]$ defining it as a scheme in $\mathbb P^n_K$ and we endow $S$ with a term order. By our techniques, we give some lower bounds for the number of irreducible components of $\HilbScheme{p(t)}{n}$ and  determine the maximal Hilbert function in every irreducible component. Moreover, we obtain some interesting results about the rationality of the components of $\HilbScheme{p(t)}{n}$, in particular for the components that contain a point corresponding to an arithmetically Cohen-Macaulay scheme in codimension two.

The notion of generic initial ideal has attracted the attention of many researchers since its first introduction. Indeed, a generic initial ideal $\Gin(I)$ of a homogeneous  ideal $I\subset S$ contains many information about the ideal $I$ and about the scheme defined by $I$ (e.g.~\cite{Gunnar}). It is noteworthy that $\Gin(I)$ can be obtained from $I$ by a flat deformation corresponding to a rational curve on the Hilbert scheme. Furthermore, the set of  generic initial ideals coincides with that of  Borel-fixed ideals, which appear as a useful tool in the investigation of the Hilbert scheme, especially for what concerns its components, already in the $60$s  \cite{H66}.
Indeed, \lq\lq every component and intersection of components of a Hilbert scheme contains at least one Borel-fixed ideal\rq\rq\ \cite{R1}. This property follows essentially by two facts: the generic initial ideals are Borel-fixed and the irreducible components of $\HilbScheme{p(t)}{n}$ are invariant under the action of the general linear group $\mathrm{GL}:=\mathrm{GL}_K(n+1)$ induced by the action on $S$.  

Every component of $\HilbScheme{p(t)}{n}$ can contain more than one Borel-fixed ideal. Nevertheless, for any given term order in $S$ and  for every irreducible component $Y$ of $\HilbScheme{p(t)}{n}$, we identify a special Borel-fixed ideal $\mathbf{G}_Y$, which we call {\em double-generic initial ideal}, and which gives us information about $Y$. Roughly speaking, $\mathbf{G}_Y$ is the \lq\lq generic initial ideal of the generic (and the general)  point of $Y$\rq\rq (see Definition \ref{def:gigi}).
The introduction and the investigation of the notion of double-generic initial ideal have been inspired by the ideas of both Gr\"obner strata and marked schemes, which support the main results of this paper, although they do not explicitely appear. In fact, some of our examples have been obtained  by the available computational methods already developed in the study of marked schemes by several authors. 

It is not easy to get a generic initial ideal: a deterministic computation using parameters can be very heavy, while random changes of coordinates gives a non-certain result.  The double-generic initial ideal overcomes these difficulties, since it is easy to individuate $\mathbf{G}_Y$ starting from the list of  Borel-fixed ideals on $Y$: it is sufficient to detect the maximum in this list w.r.t.~a suitable order on the Borel-fixed  ideals of $Y$ (see Theorem \ref{th:minoreminore} and Corollary \ref{ginmaggiore}).

Along the paper, we consider the classical scheme-theoretic embedding of the Hilbert scheme  $\HilbScheme{p(t)}{n}$ in the Grassmannian $\GrassScheme{S_m}{q}$, where $S_m$ is the vector space of the  homogeneous polynomials in $S$  of degree $m$, for $m$ a sufficiently large degree,  and $q:={n+m \choose n}-p(m)$. More precisely, it is sufficient to take $m\geq r$, where $r$ is the Gotzmann number of the Hilbert polynomial $p(t)$ (for more details see for instance \cite{CS}). Since in turn the Grassmannian $\GrassScheme{S_m}{q}$ can be embedded in $\mathbb P(\wedge^q S_m)$ via the Pl\"ucker embedding, a point $V$ of $\HilbScheme{p(t)}{n}$ can be identified with a non-zero totally decomposable tensor, i.e.~an {\it extensor} $f_1\wedge\dots\wedge f_q$, where $f_1,\dots,f_q\in S_m$ are linearly independent polynomials such that the ideal \hskip -0.5mm $(f_1,\dots,f_q)$ defines the projective subscheme corresponding to~$V$. 

In the above setting, we follow the approach presented by D.~Eisenbud in his book \cite{Ei} to deal with the generic initial ideal by means of a suitable total order on the terms of $\wedge^q S_m$, depending  on the  term order  $\prec $ on $S$ (see \eqref{ordEisenbud}). Thus, we associate to every subset $W$ of $\GrassScheme{S_m}{q}$ a suitable set of terms in $\wedge^q S_m$ called $\Delta$-support of $W$ (see Definition \ref{def:support}), and then introduce the initial extensor $\IN(W)$ of $W$ as the maximum of the $\Delta$-support of $W$. Further,  we also introduce the generic initial extensor $\GIN(W)$ of $W$ as the maximum of the $\Delta$-support of the orbit of $W$ under the usual action of $\mathrm{GL}$ on $\GrassScheme{S_m}{q}$  (see Definition \ref{def:extIn}).  We prove that $\IN(W)$ and $\GIN(W)$ do not change when $W$ is replaced by its closure $\overline{W}$ (Proposition \ref{prop:chiusura}). In particular, if $W$ is closed and irreducible, $\IN(W)$ and $\GIN(W)$ can be read as the initial extensor and the generic initial extensor of either the generic point of $W$ or the set of closed points of $W$  (Proposition \ref{prop:chiusura}, Remark \ref{rem:generic point}). Moreover, exploiting the analogous property for ideals given in \cite[Theorem 15.18]{Ei}, we prove that $\GIN(W)$  is fixed by the Borel subgroup of $\mathrm{GL}$, up to multiplication by a non-null element of $K$ (Theorem \ref{th:eisenbud} and Corollary \ref{cor:chiusura3}). 

In Section 3, we focus our attention on the subsets $W$ of $\GrassScheme{S_m}{q}$ that are closed and stable under the action of $\mathrm{GL}$ and prove that  $\IN(W)$ and $\GIN(W)$ coincide  and the corresponding point of the Grassmannian belongs to $W$ (Proposition \ref{prop:componenti}). If   $W$ is closed and irreducible, then there is a dense open subset of $W$ consisting of points having $\GIN(W)$ as generic initial extensor (Proposition~\ref{lemma:UL}).

In Section 4, we concentrate on subsets of the Hilbert scheme. We prove that there is a perfect correspondence between the   notions of initial  and generic initial extensor of any point $V$ of $\HilbScheme{p(t)}{n}$ and those of initial and generic initial ideal of any ideal defining  $V$ as a subscheme of $\PP^n_K$  (Theorem \ref{gin appartiene} and Corollary \ref{appartenenza}). Furthermore, we prove that if $Y\subseteq \HilbScheme{p(t)}{n}$ is irreducible, closed and invariant under the action of $\mathrm{GL}$, then the ideal associated to $\GIN(Y)$ does not depend on the chosen $m\geq r$ for the embedding in a Grassmannian scheme (Corollaries \ref{appartenenza} and \ref{prop:GIGI}). This allows us to define the double generic initial ideal of $Y$ (Definition \ref{def:gigi}). We also present some relevant subsets of $\HilbScheme{p(t)}{n}$ which are irreducible, closed and invariant under the action of $\mathrm{GL}$ (Examples \ref{ex:singular} and~\ref{ex:gore e punti}).

In Section 5, we introduce a suitable partial order $\prec\!\!\prec$ on the terms of  $\wedge^q S_m$  and prove that  the initial extensor and the generic initial extensor of a  closed irreducible subset $W\subseteq \GrassScheme{S_m}{q}$ are, respectively,  the maxima of the $\Delta$-supports of $W$ and of its orbit with respect to this partial order (see Definition~\ref{precprec} and Theorem~\ref{th:minoreminore}).  Although a partial order might appear less convenient  than a total order, this feature of $\prec\!\!\prec$ is in fact a crucial point of the paper, which we exploit in order to obtain some relevant applications. We also explore some of the properties of this partial order, in particular when we consider either the degrevlex term order or a constant Hilbert polynomial (Propositions \ref{prop: m cresce} and \ref{m cresce con polinomio costante}).

Finally, in Section \ref{sec:applications}, we present some interesting applications of the previous results. First, we point out a necessary condition for a Borel term to correspond to a point of a given irreducible component of a Hilbert scheme (Proposition \ref{cor:condizione necessaria}). Then, we obtain lower bounds on the number of irreducible components of  $\HilbScheme{p(t)}{n}$ simply counting the maximal elements with respect to $\prec\!\!\prec$ among the extensor terms corresponding to the Borel-fixed ideals in $\HilbScheme{p(t)}{n}$. We observe that this bound depends on the chosen term order on $S$ (Proposition \ref{prop:numero componenti} and Example \ref{ex:4 componenti}). The list of {\em all} the saturated Borel-fixed ideals in $\HilbScheme{p(t)}{n}$ can be obtained by the algorithms presented in \cite{CLMR,PL} in the characteristic  zero case, and in \cite{B2014} for every characteristic.

Recall that, for every irreducible, closed subset $Y\subseteq \HilbScheme{p(t)}{n}$ that is stable under the action of $\mathrm{GL}$, there is a maximum among the Hilbert functions of its points (Remark \ref{rem:maximum}). We prove that this maximum is reached by the point corresponding to the double-generic initial ideal $\mathbf{G}_Y$, hence by the maximum with respect to $\prec\!\!\prec$, when we choose the degrevlex term order on $S$ (Theorem \ref{funzione massima}). We conjecture that an analogous result holds for minimal Hilbert functions, when the deglex term order is chosen.
 
We conclude by investigating the rationality of some components of a Hilbert scheme. We prove that, if $Y$ is an isolated irreducible component of $\HilbScheme{p(t)}{n}$ and $\GIN(Y)$ corresponds to a smooth point in $Y$, then $Y$ is rational (Theorem \ref{razionale}). In the final Example \ref{ex:tommasino}, we exhibit a Hilbert scheme having two components with the same double-generic initial ideal which is smooth on each component, but singular on the Hilbert scheme. Thus, by Theorem \ref{razionale} both these components are rational. As a relevant consequence of Theorem \ref{razionale}, by exploiting \cite[Th\'{e}or\`{e}me 2(i)]{Elli} we prove that every component of $\HilbScheme{p(t)}{n}$ containing a Cohen-Macaulay point of codimension $2$ is rational (Corollary \ref{prop:razionale aCM codim 2}).  More precisely, the Cohen-Macaulay locus of such a component is the orbit of an open subset isomorphic to an affine space under the action of $\mathrm{GL}$. This improves results of J. Fogarty \cite[Theorem 2.4 and Corollary 2.6]{Fo} and R. Treger \cite[Theorem 2.6]{TregerIII} by independent arguments. All these results on rationality give a partial answer to one of the open questions on Hilbert schemes collected at the AIM workshop Components of Hilbert Schemes (Palo Alto, July 19 to 23, 2010) \cite[Problem 1.45]{AimQ}.


\section{Generalities}
\label{sec:generalities}

Let $S:=K[x_0,\dots,x_n]$ be the polynomial ring over an infinite field $K$. For every integer $m$, $S_m$ denotes the homogeneous component of degree $m$ of $S$; if $A\subseteq S$, then $A_m$ denotes $A\cap S_m$. Elements and ideals in $S$ will be always assumed to be homogeneous. 
A term $\tau$ of $S$ is a power product $\tau = x_0^{\alpha_0}\cdot\dots\cdot x_n^{\alpha_n}$. 
The set of all the terms of $S$ will be denoted by $\mathbb T$. 

We denote by $\prec$ a given term order in $S$ and  assume that $x_0\prec x_1\prec\dots\prec x_n$. In our setting, if $\tau=x_0^{\alpha_0}\dots x_n^{\alpha_n}$ and $\sigma =x_0^{\beta_0}\dots x_n^{\beta_n}$ are two terms in $S$ of the same degree, then 
\begin{itemize}
\item if $\prec$ is the deglex term order, then $\tau\prec \sigma$ if and only if $\alpha_k<\beta_k$, where $k:=\max\{i\in \{0,\dots,n\} \ \vert \ \alpha_i\not= \beta_i\}$; 
\item if $\prec$ is the degrevlex term order, then $\tau\prec \sigma$ if and only if $\alpha_k>\beta_k$, where $k:=\min\{i\in \{0,\dots,n\} \ \vert \ \alpha_i\not= \beta_i\}$.
\end{itemize}

If $J$ is a monomial ideal in $S$, $B_J$ denotes the monomial basis of $J$, i.e.~the set of the terms that are minimal generators of $J$. For any non-zero polynomial $f \in S$, the \textit{support} $\Supp(f)$ of $f$ is the set of the terms that appear in $f$ with a non-zero coefficient. 

The maximal term occurring in the support of $f$ with respect to $\prec$ (w.r.t.~$\prec$) is called the {\em initial term} of $f$ w.r.t.~$\prec$ and denoted by $\mathrm{in}_\prec(f)$. If $I$ is an ideal in $S$, the {\em initial ideal} $\mathrm{in}_\prec(I)$ of $I$ w.r.t.~$\prec$ is the ideal generated by the initial terms of the polynomials in $I$. When there is no ambiguity, we will write $\mathrm{in}(f)$ and $\mathrm{in}(I)$ in place of $\mathrm{in}_\prec(f)$ and $\mathrm{in}_\prec(I)$.

A set $\{f_1,\dots,f_t\}$ of monic polynomials of an ideal $I$ is the {\em reduced Gr\"obner basis} of $I$, w.r.t.~$\prec$, if 
$\mathrm{in}(I)=(\mathrm{in}(f_1),\dots,\mathrm{in}(f_t))$ and no term in  $\Supp(f_i)\setminus \{\mathrm{in}(f_i)\}$  belongs to $\mathrm{in}(I)$. 
We refer to \cite{Mora2005} for an extended treatment of the theory of
Gr\"obner bases and related topics. 

\medskip
We consider the general linear group $\mathrm{GL}_K(n+1)$ ($\mathrm{GL}$, for short) of the invertible matrices of order $n+1$ with entries in $K$. If $g=(g_{ij})_{i,j \in \{0,\dots,n\}}$ is a matrix in $\mathrm{GL}$, $g$ acts on the polynomials of $S$ in the following way
\[
\begin{array}{rcl}
g:S&\rightarrow & S\\
f(x_0,\dots,x_n)&\mapsto & g(f)=f(\sum_{j=0}^n g_{0j}x_j,\dots, \sum_{j=0}^n g_{n j}x_j)
\end{array}
\]
For every subset $V \subseteq S$, we let $g(V)=\{g(f(x_0,\dots,x_n))\vert  f \in V\}$.

\medskip
If $I$ is an ideal of $S$, we can consider $\mathrm{in}(g(I))$. It is well-known that there is an open subset $\mathcal A\not=\emptyset$ of $\mathrm{GL}$ such that, for every $g\in \mathcal A$, $\mathrm{in}(g(I))$ is a constant monomial ideal called the {\em generic initial ideal} of $I$ and denoted by $\mathrm{gin}_\prec(I)$ (or  $\mathrm{gin}(I)$ when there is no ambiguity).
In our setting, $\mathrm{gin}(I)$ is fixed under the action of the Borel subgroup of upper-triangular invertible matrices, hence it is a Borel-fixed 
ideal (Borel, for short) (see Galligo's Theorem \cite{GA2} for $\mathrm{char}(K)=0$ and \cite[Proposition 1]{BS} for a positive characteristic).

Note that  for every degree $m$:
\begin{equation}\label{eq:tagliInGin}
\In(I_{ m})=\In(I)_{m}\text{ and }\Gin(I_{ m})=\Gin(I)_{m}.
\end{equation}

The saturation of the ideal $I\subset S$ is the ideal $I^\sat:=\cup_{k\geq 0}(I:(x_0,\dots,x_n)^k)$ and $I$ is \emph{saturated} if $I=I^\sat$. The schemes $\Proj(S/I)$ and $\Proj(S/I')$ are equal as subschemes of $\PP^n_K$ if and only if $I$ and $I'$ have the same saturation.

Given a homogeneous ideal $I$ in $S$, we refer to \cite[Chapters 1 and 4]{Ei-syzygies} for the definitions of the Hilbert function (denoted by $H_{S/I}$),  Hilbert polynomial and Castelnuovo-Mumford regularity, or simply {regularity} (denoted by $\reg(I)$). Here, we recall that the regularity $\reg(X)$ of the scheme $X=\Proj(S/I)$ is the regularity of the ideal $I^{sat}$ and $\reg(I)\geq \reg(I^{sat})$. Moreover, for every $m\geq \reg(I)$, we say that $I$ is {\em $m$-regular} and the Hilbert function $H_{S/I}$ satisfies $H_{S/I}(m)=p(m)$, where $p(t)$ is its Hilbert polynomial. In this case, we will also say that $p(t)$ is the Hilbert polynomial of $I$.
Recall that, if $m\geq \reg(I)$, then $\Proj(S/I)$ can be completely recovered by the $K$-vector space $I_m$, since $(I_m)^{sat}=I^{sat}$. If $\mathrm{char}(K)=0$, the regularity of a Borel-fixed ideal is the maximum of the degrees of its minimal generators.
The initial ideal and the generic initial ideal of an ideal $I$ have the same Hilbert function as $I$. However, their regularities satisfy the  inequa\-lities $\reg(\mathrm{in}(I))\geq \reg(I)$ and $\reg(\mathrm{gin}(I))\geq \reg(I)$, which sometimes are strict. Moreover, the initial ideal and the generic initial ideal of a saturated ideal can be no more saturated. However, if $\prec$ is the degrevlex term order, then $I$ and $\Gin(I)$ share the same regularity and each of them is saturated if the other is (see \cite{BS2}).

\begin{example}\label{ex:regsale}
Let $I$ be the ideal $(x_2^2,x_1x_2+x_0^2)\subset  K[x_0,x_1,x_2]$, with $\mathrm{char}(K)=0$ (see \cite{Mayes}). The ideal $I$ is saturated and $\reg(I)=3$, while, for every term order on $S$, $\In(I)=(x_2^2, x_1x_2,x_0^2x_2,x_0^4)$ is not saturated, and $\reg(\In(I))=4$. If $\prec$ is the degrevlex term order, then $\mathrm{gin}(I)=(x_2^2, x_1x_2,x_1^3)$ is a saturated ideal with regularity $3$. If $\prec$ is the deglex term order, then $\Gin(I)=(x_2^2,x_1 x_2,x_0^2x_2,x_1^4)$ is not saturated with $\reg(\Gin(I))=4$. 
\end{example}


\section{Initial and generic initial extensors}
\label{sec:in e gin}

Let $\prec$ be a term order on $S$ and $m$ a positive integer. In the present section, we consider $S_m$ as a $K$-vector space.

The \emph{initial space} of a $K$-vector space $V\subset S_m$ is the $K$-vector space $\mathrm{in}(V):=\langle \mathrm{in}(f)\vert f\in V\rangle$. There is an open subset $\mathcal A\not=\emptyset$ of $\mathrm{GL}$ such that, for every $g\in \mathcal A$, $\mathrm{in}(g(V))$ is constant. This constant initial space is called the \emph{generic initial space} of $V$ and is denoted by $\mathrm{gin}(V)$ (see \cite{Gunnar} and the references therein).

Let $q$ be a positive integer and $\GrassScheme{S_m}{q}$ the Grassmannian of subspaces of $S_m$ of dimension~$q$. We consider $\GrassScheme{S_m}{q}$ as a subscheme embedded in the projective space $\mathbb P(\wedge^q S_m)$ through the Pl\"ucker embedding (for example, see \cite{Harris1995}). 

\begin{definition}
An \emph{extensor} (of step $q $) on $S_m$ is a non-zero element of $\wedge^q S_m$ of the form $f_1\wedge\dots\wedge f_q$, with $f_1,\dots,f_q\in S_m$.
\end{definition}

Note that an element $f_1\wedge \dots \wedge f_q \in \wedge^q S_m$ vanishes whenever the vector space generated by $f_1,\dots,f_q$ has dimension lower than $q$.

Following \cite[Section 15.9]{Ei}, we say that an \emph{extensor term} (or simply a \emph{term}) in $\wedge^{q} S_m$ is an extensor  of type $\tau_1\wedge \dots\wedge \tau_{q}$, with $\tau_i\in S_m\cap \mathbb T$; furthermore, we say that a term of $\wedge^q S_m$ is a \emph{normal expression} if $\tau_1\succ \dots\succ \tau_{q}$. We denote by $\mathbf{T}_{S_m}^q$ the set of the normal expression terms and from now on, whenever we consider a term $\tau_1\wedge \dots \wedge \tau_q\in \wedge^q S_m$, we assume that it belongs to $\mathbf T_{S_m}^q$. Furthermore, $\mathbf T_{S_m}^q$ is the $K$-vector basis we always consider for the $K$-vector space $\wedge^q S_m$. We can compare the terms in $\mathbf T_{S_m}^q$ lexicographically according to $\prec$, in the following way
\begin{multline} \label{ordEisenbud}
\tau_1\wedge \dots\wedge \tau_{q}\prec \sigma_1\wedge \dots\wedge \sigma_{q} \Leftrightarrow\\
 \exists \ j\in \lbrace 1,\dots,q\rbrace \ : \ \tau_i=\sigma_i, \ \forall \ i<j, \text{ and } \tau_j\prec \sigma_j.
\end{multline}

In this setting, for every $L\in \mathbf T_{S_m}^q$, there is a unique Pl\"ucker coordinate $\Delta_L$ on $\GrassScheme{S_m}{q}$ corresponding to $L$, and vice versa. Moreover, $\GrassScheme{S_m}{q}= \Proj(K[\Delta_L : L\in \mathbf T_{S_m}^q]) \ =: \ \Proj(K[\Delta])$. 
Therefore, if $V=\langle f_1,\dots,f_q\rangle\subseteq S_m$ is a $K$-vector space of dimension $q$, the extensor $f_1\wedge \dots \wedge f_{q}$ has the unique writing $\sum_{L\in \mathbf T_{S_m}^q} c_L L$, with $c_L= \Delta_L(V)\in K$.

\medskip 
For every $L=\tau_1\wedge\dots\wedge \tau_q \in \mathbf T_{S_m}^q$, we will denote by $U_L$ the standard open set of $\GrassScheme{S_m}{q}$ corresponding to $\Delta_L$, namely the locus of points in $\GrassScheme{S_m}{q}$ where $\Delta_L$ is invertible. Moreover, we will denote by $\vs{L}$ the vector space $\langle \tau_1,\dots, \tau_q \rangle$ in $\GrassScheme{S_m}{q}$.

\begin{definition}\label{def:support}
If $W$ is a non-empty subset of $\GrassScheme{S_m}{q}$,  the \emph{$\Delta$-support} of $W$ is the following subset of $\mathbf T_{S_m}^q$:
$$\supp(W):=\lbrace L \in \mathbf T_{S_m}^q\ \vert\ U_{L}\cap W \not=\emptyset\rbrace.$$
If $W=\lbrace V\rbrace$, we simply write  $\supp(V)$ for $\supp(\lbrace V\rbrace)$.

We will apply this definition and the related ones also to subschemes $W$ of  $\GrassScheme{S_m}{q}$, meaning that the  $\Delta$-support of a scheme is that of its underlying set of points (see also Example \ref{nonridotto})
\end{definition}

If $W$ is a non-empty subset of $\GrassScheme{S_m}{q}$, then its $\Delta$-support is non-empty. From now, we consider subsets of $\GrassScheme{S_m}{q}$ that are non-empty.

\begin{proposition}\label{prop:chiudere}
Let $W$ be a subset of $\GrassScheme{S_m}{q}$.
\begin{enumerate}[(i)]
\item \label{prop:chiudere_i} If  $\overline W$ is the closure of $W$, then $\supp(W)=\supp(\overline W)$.
\item \label{prop:chiudere_ii} If $W$ is closed and ${\widetilde W}$ is the set of  its closed points, then $\supp(W)=\supp(\widetilde W)$.
\item \label{prop:chiudere_iii} If $W$ is closed and irreducible and $V$ is its generic point, then $\supp(W)$ $=\supp(V)$.
\end{enumerate}
\end{proposition}

\begin{proof}
\eqref{prop:chiudere_i} It is immediate that $\supp(W)\subseteq \supp(\overline W)$, because $ W\subseteq  \overline W$. We now prove the other inclusion. If $L$ belongs to $\supp(\overline W)$, then there is at least a point $V$ in $U_L\cap \overline W$. Thus, every open neighbourhood of $V$ meets $W$ non-trivially, in particular $U_L$ meets $W$ non-trivially. Hence, $L$ belongs to $\supp(W)$.

Items \eqref{prop:chiudere_ii} and \eqref{prop:chiudere_iii} are straightforward consequence of \eqref{prop:chiudere_i}, because $\overline{\widetilde W}=W$ (for example, see  \cite[Chapter V, Section 3.4, Theorem 3]{bour}) and $\overline{\{V\}}={W}$, in the respective hypotheses. 
\end{proof}

The action of $\mathrm{GL}$  on $S$ defined in Section \ref{sec:generalities} induces  an action on $\wedge^q S_m$ in the following natural way: if $H=\tau_1\wedge \dots \wedge \tau_q$  is a term in $ \mathbf T_{S_m}^q$, we set $g(H)=g(\tau_1)\wedge\dots \wedge g(\tau_q)\in \wedge^q S_m$ where $g\in \mathrm{GL}$ and then extend the action to every element in $\wedge^q S_m$ by linearity. Note that in general $g(H)$ does not need to be a term. 

In a similar natural way, we obtain an action of $\mathrm{GL}$ on $K[\Delta]$ and on $\GrassScheme{S_m}{q}$ (see also \cite[Example 10.18]{Harris1995}): for every $g\in \mathrm{GL}$ and for every $H\in \mathbf T_{S_m}^q$ (hence, every $\Delta_H\in \Delta$), if $g(H)=\sum_{L\in \mathbf T_{S_m}^q} c_L L$, then we set $g(\Delta_H):=\sum_{L\in \mathbf T_{S_m}^q} c_L \Delta_L$. Thus, for every point $V$ of $\GrassScheme{S_m}{q}$ and for every element $g$ of $\mathrm{GL}$, by $g(V)$ we mean the point of $\Proj(K[\Delta])$ corresponding to the prime ideal $g(\mathfrak a)$, with $\mathfrak a$ the prime ideal defining $V$. If $V$ is a $K$-point of $\GrassScheme{S_m}{q}$, i.e.~$V=\langle f_1,\dots,f_q\rangle$ with $f_i\in S_m$, then  this action on  $V$ is  exactly  the usual action of $\mathrm{GL}$ on the polynomials generating $V$ as a $K$-vector space.

In this context, the orbit of $V\in\GrassScheme{S_m}{q}$ is the set $\O(V):=\lbrace g(V)\ \vert\ g\in \mathrm{GL} \rbrace$. If $W$ is a subset of $\GrassScheme{S_m}{q}$, the orbit of $W$ is $\O(W):=\cup_{V\in W} \O(V)$.

\begin{definition}\label{def:extIn}
If $W$ is a subset of $\GrassScheme{S_m}{q}$, the \emph{initial extensor} of $W$ is
$$\IN(W):=\max_{\prec} \supp(W)$$
and the \emph{generic initial extensor} of $W$ is 
\[
\GIN(W):=\IN(\O(W))=\max_{\prec} \supp(\O(W)),
\]
where the maximum is taken w.r.t.~to the order $\prec$, as defined in \eqref{ordEisenbud}. 
If $W=\lbrace V\rbrace$, we simply write $\IN(V)$ for $\IN(\lbrace V\rbrace)$ and $\GIN(V)$ for $\GIN(\lbrace V\rbrace)$. 
\end{definition}

\begin{remark}\label{rem:ginpiugrande}
Note that, for every  subset $W $ of $\GrassScheme{S_m}{q}$, we have $\IN(W)\preceq\GIN(W)$ because $\supp(W)\subseteq \supp(\O(W))$.
\end{remark}

\begin{definition}\label{borel nel wedge}
We say that a term $L\in \mathbf T_{S_m}^q$ is {\em Borel-fixed} ({\em Borel}, for short) if $g(L)\in \langle L\rangle$ for every upper-triangular matrix $g\in \mathrm{GL}$. We will denote by $\mathbf B_{S_m}^q$ the set of Borel terms.
\end{definition}

Note that a term $L\in\mathbf T_{S_m}^q$ is Borel if and only if the ideal  generated by $\vs{L}$ in $S$ is Borel.

\begin{theorem}\label{th:eisenbud}
Let $V=\langle f_1,\ldots,f_q\rangle$ be a $K$-point of $\GrassScheme{S_m}{q}$ and let $f_1\wedge \dots \wedge f_{q}=\sum_{L\in \mathbf T_{S_m}^q} c_L  L$. Then: 
\begin{enumerate}[(i)]
\item\label{th:eisenbud_i} $\supp(V)=\{L\in \mathbf T_{S_m}^q\ \vert\ c_L\neq 0\}$. 
\item  $\wedge^q \mathrm{in}(V)= \langle \IN(V)\rangle  $, or equivalently $\mathrm{in}(V)= \vs{\IN(V)}$.
\item\label{th:eisenbud_iii} $\wedge^q \mathrm{gin}(V)=\langle \GIN(V)\rangle$, or equivalently $\mathrm{gin}(V)= \vs{\GIN(V)}$, and $\GIN(V)$ is Borel.
\end{enumerate}
\end{theorem}

\begin{proof}
\begin{enumerate}[(i)]
\item As already observed, for a $K$-point $V$ and a term $L$, we have $\Delta_L(V)=c_L \in K$. Hence, $V$ belongs to $U_L$ if and only if $\Delta_L(V)$ is invertible in $K$, thus if and only if $\Delta_L(V)$ is non-zero.
\item If $\mathrm{in}(V)=\langle \tau_1,\dots,\tau_q\rangle$, with $\tau_1\succ\dots \succ \tau_q$, we can assume that $f_1,\dots,f_q$ are polynomials such that $\mathrm{in}(f_i)=\tau_i$, for every $i\in\lbrace 1,\dots,q\rbrace$. Then, exploiting item \eqref{th:eisenbud_i} we see that 
$\IN(V)= \max_{\prec}\supp(\langle f_1,\dots, f_q\rangle) = \tau_1\wedge \dots\wedge \tau_q$.
\item We immediately obtain the equality $\wedge^q \mathrm{gin}(V)=\langle \GIN(V)\rangle$ and can conclude by \cite[Theorem 15.18]{Ei}, taking into account also formula \eqref{eq:tagliInGin}.
\end{enumerate}
\end{proof}

By the following example we underline that the definition of $\Delta$-support of a subscheme $W$ of $\GrassScheme{S_m}{q}$ does not depend on the possible non-reduced structure of $W$, 
i.e.~the $\Delta$-support of $W$ coincides with the $\Delta$-support of $W^{red}$. Moreover, whereas our definition of $\Delta$-support can be applied to all the subschemes $W$ of a Grassmannian,
the characterization of $\Delta$-support given in Theorem \ref{th:eisenbud}\eqref{th:eisenbud_i} cannot be extended to every $W$, even when $W$ can be defined by an extensor.

\begin{example}\label{nonridotto}
Let $S=K[x_0,x_1,x_2]$ be endowed with the degrevlex term order. Let us consider the non-reduced closed subscheme $W \subset \GrassScheme{S_2}{2}=\Proj(K[\Delta])$ defined by the ideal $I$ that is generated by $\Delta_L^2$, where $L=x_2^2\wedge x_1x_2$, and by all the other Pl\"ucker coordinates except $\Delta_L$ and $\Delta_{L'}$, where $L'=x_1x_2\wedge x_1^2$. Note that $W$ is a non-empty subscheme of the Grassmannian, since the radical of $I$ is not irrelevant. 
Indeed, $W$ is a double structure on the closed $K$-point $V'$ with Pl\"ucker coordinates all equal to $0$, except $\Delta_{L'}$. Then, $W$ does not intersect the standard open subsets of the Grassmannian, except $U_{L'}$. Thus, we obtain $\supp(W)=\supp(V')=\{{L'} \}$ and $\IN(W)=\IN(V')=L'$.

Being a double structure over a closed $K$-point, $W$ is isomorphic to the scheme $\mathrm{Spec} (K[\varepsilon])$ (where $\varepsilon^2=0$). Then, using the funtorial language, we can see $W$ as an element of  $\GrassFunctor{S_2}{2}(K[\varepsilon])$, where $\GrassFunctor{S_2}{2}$ denotes the Grassmann functor (e.g., see \cite[formula (2.1) and Section 5]{BLMR}).

Notice that $W$ is a special element of $\GrassFunctor{S_2}{2}(K[\varepsilon])$, since it is  given by the  rank 2  free (not only locally free) submodule of $S_2\otimes_K K[\varepsilon]$  generated by $f_1=\varepsilon x_2^2-x_1^2$ and $f_2=x_1x_2$.  Therefore,  we can identify it with the extensor   in $\wedge^2 (S_2\otimes_K K[\varepsilon])$
$$f_1\wedge f_2=(\varepsilon x_2^2-x_1^2) \wedge x_1x_2=\varepsilon (x_2^2 \wedge x_1x_2)-(x_1^2 \wedge x_1x_2)=\varepsilon L+L'.$$
Now we can see that Theorem \ref{th:eisenbud}\eqref{th:eisenbud_i} cannot be extended to this case; in fact the coefficient of $L$ in  $\varepsilon L+L'$ is $\varepsilon \neq0$, but $L$ does not belong to $\supp(W)$.
\end{example}

\begin{proposition}\label{prop:chiusura}
Let $W$ be a subset of $\GrassScheme{S_m}{q}$. 
\begin{enumerate}[(i)]
\item \label{prop:chiusura_i} If $\overline W$ is the closure of $W$, then $\IN(W)=\IN(\overline W)$ and $\GIN(W)=\GIN(\overline W)$.

\item \label{prop:chiusura_ii} If $W$ is closed and   ${\widetilde W}$ is the set of  its closed points, then $\IN(W)=\IN(\widetilde W)$ and $\GIN(W)=\GIN(\widetilde W)$. 

\item \label{prop:chiusura_iii} If  $W$ is    closed and irreducible  and $V$ is its generic point,  then $\IN(W)=\IN(V)$ and $\GIN(W)=\GIN(V)$. 
\end{enumerate}
\end{proposition}

\begin{proof}
For what concerns the initial extensor, the three statements directly follow from Proposition \ref{prop:chiudere}. For what concerns the generic initial extensor, the statements in \eqref{prop:chiusura_ii} and \eqref{prop:chiusura_iii} are consequences of \eqref{prop:chiusura_i}, since $\overline{\widetilde W}=W$ and $\overline{\{V\}}={W}$ in the respective hypotheses.
Then, it remains to prove  the statement about the generic inital extensor in \eqref{prop:chiusura_i}. It is sufficient to show that $\supp(\O(W))=\supp(\O(\overline W))$.  We only prove the non-obvious inclusion.

Let $L$ be any term in $\supp(\O(\overline W))$. By definition, $U_L\cap \O(\overline W)$ is not empty. More precisely, there are an element $g\in \mathrm{GL}$ and a point $V_1\in \overline W$ such that $g(V_1)\in U_L$. Then, $g(V_1)$ belongs to $U_L\cap \overline{g(W)}$, since  $g(\overline W)=\overline{g(W)}$. By the definition of closure, this implies $U_L\cap g(W)\neq \emptyset$. Hence,  $U_L\cap \O(W)\neq \emptyset$ and  $L\in \supp(\O(W))$.
\end{proof}

\begin{remark}\label{rem:generic point} \ 

\begin{enumerate}[(i)]
\item In many cases, we will identify a closed subset $W$ of $\GrassScheme{S_m}{q}$ either with the set $\widetilde W$ of its closed points or, if $W$ is also irreducible, with its generic point. Indeed, by Proposition \ref{prop:chiusura}, $\IN(W)$ and $\GIN(W)$ can be also read as the initial extensor and the generic initial extensor either of $\widetilde W$ or, if $W$ is irreducible, of the generic point of $W$.  These facts will be useful because, up to an extension of the base field $K$ to the residue field $K_V$ (where $V$ is either a suitable point in $\widetilde W$ or the generic point of $W$), they will allow to reduce our arguments to the case of a rational point. Note that, being $K$ infinite, $\mathrm{GL}$ is Zariski dense in $\mathrm{GL}\otimes_K K'$ for every extension field $K'$ of $K$, hence the computation of the generic initial extensor does not change after an extension of the base field. Furthermore, if $K$ is algebraically closed and we can identify $W$ with $\widetilde W$, then we do not need to extend the base field.
\item If $W$ is closed and irreducible, we can read $\IN(W)$ and $\GIN(W)$ as the initial extensor and the generic initial extensor of a {\it general point} in $W$. In fact, consider the two non-empty open subsets $W':=W \cap U_{\IN(W)}$ and $W'':=W\cap U_{\GIN(W)}$ of $W$.
The sets $\widetilde{W'}$ and $\widetilde{W''}$ of the closed points in $W'$ and $W''$, respectively, are both dense in $W$.  By construction, we have $\IN(V)=\IN(W)$ for every point $V\in W'$ and $\GIN(V)=\GIN(W)$ for every point $V\in W''$.
\end{enumerate}
\end{remark}

\begin{corollary} \label{cor:chiusura3}
Let $W$ be a subset of $\GrassScheme{S_m}{q}$. Then, $\GIN(W)$ belongs to $\mathbf B_{S_m}^q$.
\end{corollary}

\begin{proof} 
By Proposition \ref{prop:chiusura} we can assume that $W$ is closed. As observed in Remark \ref{rem:generic point}, we can replace $W$ by a suitable closed point $V$ of $W$, which can be considered as a $K$-point after a possible extension of the base field.
We conclude by Theorem \ref{th:eisenbud}\eqref{th:eisenbud_iii}.
\end{proof}


\section{Stable subsets under the action of GL}
\label{sec:GL}

In this section we focus our attention on the subsets $W$ of $\GrassScheme{S_m}{q}$ that are stable under the action of the group $\mathrm{GL}$, i.e.~$g(W)=W$ for every $g\in \mathrm{GL}$. 
Under this hypothesis on $W$ we see that the initial extensor and the generic initial extensor of $W$ coincide. If, moreover, we assume that $W$ is also closed and irreducible, then we obtain interesting properties of the two open subsets $W \cap \mathcal V_{\IN(W)}$ and $W \cap \mathcal U_{\IN(W)}$ of the points of $W$ having $\IN(W) = \GIN(W)$ as initial extensor and generic initial extensor, respectively (see formula \eqref{eq:V tondo con L}).

Let $V$ be a point in $\GrassScheme{S_m}{q}$. The closure $\overline{\O(V)}$ of its orbit $\O(V)$ is irreducible, because $\O(V)$ is irreducible, and is stable under the action of $\mathrm{GL}$ because $g(\mathcal O(V)) \subseteq \mathcal O(V)$ by definition of orbit, for every $g\in \mathrm{GL}$, and hence $g(\overline{\mathcal O(V)}) \subseteq \overline{g(\mathcal O(V))} \subseteq \overline{\mathcal O(V)}$.
In particular, every subset $W$ of $\GrassScheme{S_m}{q}$ that is stable under the action of $\mathrm{GL}$ is a disjoint union of  orbits of its points under the action of $\mathrm{GL}$. If, moreover, $W$ is also closed, it is the union of the closure of these orbits.
 
\begin{proposition}\label{prop:componenti}
Let $W\subseteq \GrassScheme{S_m}{q}$ be closed and stable under the action of $\mathrm{GL}$. Then:
\begin{enumerate}[(i)] 
\item\label{prop:componenti_i}  $\IN(W)=\GIN(W)$.
\item\label{prop:componenti_0} For every $W' \subseteq W$, both $\vs{\IN(W')}$ and $\vs{\GIN(W')}$  belong to $W$.
\item\label{prop:componenti_ii} If $W$ is reducible, then its irreducible components are stable under the action of $\mathrm{GL}$.
\item\label{prop:componenti_iii} If $Y_1,\dots, Y_\ell$ are  irreducible components of $W$, then every irreducible component of $\cap_{i=1}^\ell Y_i$ is stable under the action of $\mathrm{GL}$.
\item\label{prop:componenti_iv}  Let $U\subseteq W$ be open and stable under the action of $\mathrm{GL}$; the irreducible components of $W\setminus U$ and those of the closure $\overline U$ of $U$ are stable under the action of $\mathrm{GL}$.
\end{enumerate}
\end{proposition}

\begin{proof}
\eqref{prop:componenti_i} To prove $\IN(W)=\GIN(W)$  it is enough to observe that in the present hypothesis $W=\mathcal O(W)$. 

\eqref{prop:componenti_0} By Proposition \ref{prop:chiusura}\eqref{prop:chiusura_i} and \eqref{prop:chiusura_ii}, we may assume that $W'$ is closed and choose  a closed point $V \in W'$ such that $\IN(V)=\IN(W')$. Extending the field of scalars, if necessary,  we may assume that $V$ is a  $K$-point.    Exploiting the term order, we can construct      a map $\varphi \colon \mathbb A^1_K \rightarrow   \GrassScheme{S_m}{q}$  such that $\varphi (1)=V$,$\varphi(0)=\In(V)$, and  for every $c\neq0,1$  $\varphi (c)=g_c (V)$  where $g_c \in \mathrm{GL}$ corresponds to a diagonal matrix in which the entries of the diagonal are suitable powers of $c$ (see for instance \cite{BM91}). Then we conclude, since  $W$ is closed and contains the orbits of its points. This same argument applies to $\GIN(W')$.

\eqref{prop:componenti_ii} We have to show that every irreducible component $Y$ of $W$ is stable under the action of $\mathrm{GL}$, namely that $g(Y)=Y$ for every $g\in \mathrm{GL}$. By topological arguments,
$g(Y)$ is an irreducible component of $W$. Let $V$ be a point in $Y$ not belonging to any other irreducible component of $W$. Then, $Y$ is the only irreducible component of  $W$ containing the orbit $\O(V)$. On the other hand,  $\O(V)=g(\O(V))$ is contained in $g(Y)$ and, in particular, $V \in g(Y)$. Therefore $Y=g(Y)$.

Item \eqref{prop:componenti_iii} follows directly from \eqref{prop:componenti_ii}.

We now prove \eqref{prop:componenti_iv}. For what concerns  $W\setminus U$ it is sufficient to observe that it is closed in  $\GrassScheme{S_m}{q}$ and stable under the action of $\mathrm{GL}$. Hence, we can apply \eqref{prop:componenti_ii}. Finally, consider $V\in \overline U\setminus U$. The intersection of every open neighbourhood $A$ of $V$ with $U$ is non-empty. Moreover, for every $g\in \mathrm{GL}$, $g(A)$ is an open neighbourhood of $g(V)$ and also its intersection with $U$ is non-empty, because $U$ is stable under the action of $\mathrm{GL}$. Then, $\overline U$ is stable under the action of $\mathrm{GL}$ too and we again apply \eqref{prop:componenti_ii}.
\end{proof}

For every $L\in \mathbf T_{S_m}^q$, consider the following subsets of $\GrassScheme{S_m}{q}$:
\begin{equation}\label{eq:V tondo con L}
\mathcal V_L:=\lbrace V\in \GrassScheme{S_m}{q} \ \vert\ \IN(V)=L\rbrace, \quad \mathcal U_L:=\lbrace V\in \GrassScheme{S_m}{q} \ \vert\ \GIN(V)=L\rbrace.
\end{equation}
Obviously, $\vs{L}$ belongs to $\mathcal V_L$,  and $\mathcal U_L$   is non-empty if and only if $L$ belongs to $\mathbf B_{S_m}^q$. Thus, from now, when we consider $\mathcal U_L$ we assume that $L$ is Borel. It is immediate that $\mathcal U_L$ is stable under the action of $\mathrm{GL}$, while in general $\mathcal  V_L$ is not, even when $L$ is Borel-fixed.
We also point out that $\mathcal U_L$ does not need to contain $\vs{\IN (V)}$ for every $V\in \mathcal U_L$, even when $\IN(V)$ is Borel-fixed, as the following example shows.

\begin{example}\label{ex:BorelNonGin}
Let  us assume $\mathrm{char} (K)=0$  and consider $\GrassScheme{S_2}{3}$ and the degrevlex term order on $S=K[x_0,x_1,x_2]$ with $x_0\prec x_1\prec x_2$. If we take $V=\langle x_2^2,x_0x_2,$ $x_1x_2+x_1^2\rangle$, we obtain $L:=\GIN(V)=x_2^2\wedge x_1x_2\wedge x_1^2$ and $L':=\IN(V)=x_2^2\wedge x_1x_2\wedge x_0x_2$, both elements of  $\mathbf B_{S_2}^{3}$. Hence, $V\in \mathcal U_L$,  but $\vs{\IN(V)}\notin \mathcal U_L $, because $\GIN(\vs{\IN(V)})=\IN(V)$ being $\IN(V)$  Borel-fixed. On the other hand, $V\in   \mathcal V_{L'}$, but  $\vs{\GIN(V)}\notin   \mathcal V_{L'}$.
\end{example}

For every Borel term $L$ we will now examine the relations between the three subsets $U_L$, $\mathcal V_L$ and $\mathcal U_L$ of $\GrassScheme{S_m}{q}$. 
It is obvious that $\mathcal V_L \subseteq U_L$,  by definition of initial extensor (see Definition \ref{def:extIn}), while in general $\mathcal U_L\nsubseteq U_L$. Furthermore, as shown by Example \ref{ex:BorelNonGin}, we can have both $\mathcal  V_L\nsubseteq \mathcal U_L$ and $\mathcal U_L\nsubseteq \mathcal V_L$. Some  more detailed  relations can be obtained taking into account the action of $\mathrm{GL}$.

\begin{proposition} \label{lemma:UL} \

\begin{enumerate}[(i)]
\item \label{lemma:UL_i} For every $L\in \mathbf T_{S_m}^q$,  $\mathcal V_{L}=  U_{L}\setminus \bigcup_{L'\succ L, L' \in \mathbf T_{S_m}^q } U_{L'}$.
\item \label{lemma:UL_ii} For every $L\in \mathbf B_{S_m}^q$,  $\mathcal U_{L}= \O(U_{L})\setminus  \bigcup_{L'\succ L, L' \in \mathbf B_{S_m}^q } \O(U_{L'})$.
\item \label{lemma:UL_iii} $\{ \mathcal {V}_L\}_{ L\in  \mathbf T_{S_m}^q} $ and $\{ \mathcal {U}_L\}_{ L\in  \mathbf B_{S_m}^q }$ are two stratifications of $\GrassScheme{S_m}{q}$ consisting of locally closed subsets.
\end{enumerate}
Moreover, if $W $ is a closed and irreducible subset of $\GrassScheme{S_m}{q}$ then
\begin{enumerate}[(a)]
\item\label{lemma:UL_iv} $W\cap \mathcal {V}_{\IN(W)}$ is a dense open subset of $W$, while $W\cap \mathcal {V}_{L}$ is empty if $L\succ \IN(W)$.
\item\label{lemma:UL_v} If $W$ is also stable under the action of  $\mathrm{GL}$, then  $W\cap \mathcal {U}_{\GIN(W)}=W\cap \O(\mathcal V_{\GIN(W)}) =W\cap \O(\mathcal V_{\IN(W)})$ is a dense open subset  of $W$, while $W\cap \mathcal {U}_{L}$ is empty if $L\succ \GIN(W)$.
\end{enumerate}
\end{proposition}

\begin{proof}
\eqref{lemma:UL_i} and \eqref{lemma:UL_ii} directly follow by the definition of initial extensor and generic initial extensor (see Definition \ref{def:extIn}) and by Corollary \ref{cor:chiusura3}. 

For \eqref{lemma:UL_iii} we observe that the two families $\{\mathcal V_L\}_{L\in \mathbf T_{S_m}^q}$ and $\{\mathcal U_L\}_{L\in \mathbf B_{S_m}^q}$ are partitions of the Grassmannian, since every point $V$ of $\GrassScheme{S_m}{q}$ is contained in exacly one set $ \mathcal {V}_L$, the one with $L=\IN(V)$, and    in exacly one set $ \mathcal {U}_L$, the one with $L=\GIN(V)$. Moreover,  by the previous items it follows that $ \mathcal {V}_L$ and   $ \mathcal {U}_L$ are  locally closed in $\GrassScheme{S_m}{q}$.

\eqref{lemma:UL_iv} The intersection $W\cap U_{L} $ is empty when $L\succ \IN(W)$ and $W\cap U_{\IN(W)}$ is not empty by definition of initial extensor. Then, exploiting \eqref{lemma:UL_i} we get that also $W\cap \mathcal V_{L}  =\emptyset $  when $L\succ \IN(W)$,  and $ W\cap \mathcal V_{\IN(W)} =W\cap U_{\IN(W)} $ is a dense open subset of $W$.

To prove \eqref{lemma:UL_v} we can apply the same arguments of the previous item and Proposition \ref{prop:componenti}\eqref{prop:componenti_i}. 
\end{proof}


\section{Hilbert scheme and double-generic initial ideal of a \GLin-stable subset}

Let $p(t)$ be a Hilbert polynomial and denote by $\HilbScheme{p(t)}{n}$ the Hilbert scheme para\-me\-te\-rizing the set of all subschemes with Hilbert polynomial $p(t)$ in the projective space $\mathbb P^n_K$. From now, we consider $\HilbScheme{p(t)}{n}$ as a subscheme of  $\GrassScheme{S_m}{q}$, where $m$ is an integer larger than or equal to the Gotzmann number $r$ of $p(t)$ and $q:=\binom{n+m}{m}-p(m)$ (for instance, see \cite{CS}). Moreover, let $\prec$ be a term order in $S$.

It is well-known that $\HilbScheme{p(t)}{n}$ is invariant under the action of \GLin, as a consequence of the definition of Hilbert scheme. Thus, for many aspects, we can consider the Hilbert scheme simply as a closed subscheme $W$ of the Grassmannian, also stable under the action of \GLin, and can apply all the results we have obtained in Section \ref{sec:GL} to its irreducible closed subsets that are stable under the action of $\mathrm{GL}$. There is however an important issue that comes into play when $\HilbScheme{p(t)}{n}$ is involved. Roughly speaking, it is the relation between the notions of initial and generic initial extensors and the analogous ones for ideals. Now, we investigate this relation and show that, independently of the integer $m$, there is a well-defined ideal corresponding to the generic initial extensor of a closed irreducible subset of $\HilbScheme{p(t)}{n}$ that is also stable under the action of $\mathrm{GL}$. 

\medskip
From now, a subset of $\HilbScheme{p(t)}{n}$ that is  closed, irreducible,  and stable under the action of $\mathrm{GL}$ is called a {\em \GLin-stable} subset.

\medskip
Recall that every $K$-point of $\GrassScheme{S_m}{q}$ is a $q$-dimensional $K$-vector space $V$ of $S_m$. It is natural to consider the ideal generated by $V$ in $S$ and we denote it by $I_V$. Exploiting Theorem \ref{th:eisenbud}, we now relate the initial ideal $\In(I_V)$ and the generic initial ideal $\Gin(I_V)$ to $\IN(V)$ and $\GIN(V)$, respectively. 

In general, if $V$ is any point of $\GrassScheme{S_m}{q}$, $\In(I_V)$ does not need to coincide with the ideal $I_{\vs{\IN(V)}}$ and $\Gin(I_V)$ does not need to coincide with $I_{\vs{\GIN(V)}}$, even though their homogeneous parts of degree $m$ do, as shown by the following example.

\begin{example}\label{ex:regsale2}
Let $V$ be the vector space $\langle x_2^2,x_1x_2+x_0^2\rangle  \subset  K[x_0,x_1,x_2]_2$. For any term order in  $K[x_0, x_1, x_2]$ with $x_0\prec x_1 \prec x_2$,  the initial extensor of $V$ is $\IN(  V)=x_2^2 \wedge  x_1x_2$ and the initial ideal of $I_V$ is $\In(I_V)=(x_2^2, x_1x_2,x_0^2x_2,x_0^4)$. It is evident   that $ x_2^2 $ and $  x_1x_2 $ do not generate $\In(I_V)$. Indeed, the Hilbert polynomial of  $K[x_0, x_1, x_2]/( x_2^2 ,  x_1x_2)$ is $t+2$, while the Hilbert polynomial of $K[x_0, x_1, x_2]/I_V$ and of  $K[x_0, x_1, x_2]/\In(I_V)$ is $p'(t)=4$.
\end{example}

If $V$ is a point of a Hilbert scheme, an analogous situation to Example \ref{ex:regsale2} cannot happen.

\begin{theorem}\label{gin appartiene} Let $V$ be a $K$-point of $\HilbScheme{p(t)}{n}$, and set $V_1:=\vs{\IN(V)}$ and $V_2:=\vs{\GIN(V)}$. Then, the Hilbert polynomial of $I_{V_1}$ and $I_{V_2}$ is $p(t)$, and 
$$I_{V_1}=\In(I_V)=(\In(I_V)^\sat)_{\geq m},  \qquad  I_{V_2}=\Gin(I_V)=(\Gin(I_V)^\sat)_{\geq m}.$$
\end{theorem}

\begin{proof}
Recall that $\HilbScheme{p(t)}{n}$ is a closed subscheme of  $ \GrassScheme{S_m}{q}$ and it is stable under the action of \GLin. Hence, we can apply Proposition \ref{prop:componenti}\eqref{prop:componenti_0} to $\HilbScheme{p(t)}{n}$ and get that both $V_1$ and $V_2$ belong to  $\HilbScheme{p(t)}{n}$. Therefore $I_V$, $I_{V_1}$ and $I_{V_2}$ share the same Hilbert polynomial $p(t)$.

Thinking of the points of $\HilbScheme{p(t)}{n}$ as subschemes of $\PP^n_K$, the regularity of all of them is upper bounded by the Gotzmann number of $p(t)$, in particular by $m$. As a consequence, the saturated ideal in $S$ defining a $K$-point of $\HilbScheme{p(t)}{n}$ can be completely recovered  by saturation from its homogeneous part of degree $m$, then $(\In(I_V)^\sat)_{\geq m}=((\In(I_V)^\sat)_{ m})$.

Thus, we obtain $I_{V_1}=\In(I_V)$ observing that these two ideals are generated in degree $m$, their Hilbert functions coincide in every degree $m'\geq m$, and $(I_{V_1})_m=\In(V)$ by Theorem \ref{th:eisenbud}. 
The equality $ I_{V_2}=(\Gin(I_V)^\sat)_{\geq m}$ follows by the same arguments.
\end{proof}

\begin{corollary}\label{appartenenza}
Let $W$ be a closed subset of $\HilbScheme{p(t)}{n}$. Then, the Hilbert polynomial of the ideals $I_{\vs{\IN(W)}}$ and $I_{\vs{\GIN(W)}}$ is $p(t)$.
If, moreover, $W$ is stable under the action of \GLin, then $I_{\vs{\IN(W)}}=I_{\vs{\GIN(W)}}$ is a point of $W$.
\end{corollary}

\begin{proof}
By Proposition \ref{prop:chiusura}\eqref{prop:chiusura_ii} and possibly extending $K$ to its algebraic closure as suggested in Remark \ref{rem:generic point}, there is a $K$-point $V$ of $W$ such that $\IN(V)=\IN(W)$.  Then we get the first statement  for $\IN(W)$  as a consequence of  Theorem \ref{gin appartiene}.  This same argument applies to $\GIN(W)$. The second statement directly follows by applying Proposition \ref{prop:componenti}\eqref{prop:componenti_i} and \eqref{prop:componenti_0} to the \GLin-stable subset  $\overline{\mathcal O(\{ V\})}$ of $W$. 
\end{proof}

A relevant and immediate consequence of the above results is that the points of the Hilbert scheme corresponding to the initial extensor and the generic initial extensor  of anyone of its points do not depend on the Grassmannian in which we embed the Hilbert scheme, 
recalling that, for every integer $m \geq r$ where $r$ is the Gotzmann number, the Hilbert scheme $\HilbScheme{p(t)}{n}$ can be embedded in $\GrassScheme{S_m}{q}$. 
If we take $m'\geq r$, $m'\neq m$, we replace $q$ by $q':=\tbinom{n+m'}{n}-p(m')$. 

\begin{corollary}\label{prop:GIGI} Let $Z$ be a subset of $\HilbScheme{p(t)}{n}$ and denote by $W$ and $W'$  the images of $Z$ by the embeddings of $ \HilbScheme{p(t)}{n}$ in $\GrassScheme{S_m}{q}$ and in $\GrassScheme{S_{m'}}{q'}$,  respectively, for some $m,m' \geq r$. Then, 
\[
(I_{\vs{\IN(W)}})^\sat=(I_{\vs{\IN(W')}})^\sat \ \text{ and } \ (I_{\vs{\GIN(W)}})^\sat=(I_{\vs{\GIN(W')}})^\sat.\] 
\end{corollary}

Due to Corollary \ref{prop:GIGI}, we can finally give the following definition.

\begin{definition}\label{def:gigi}
Let $Y$ be a {\em \GLin-stable} subset of $ \HilbScheme{p(t)}{n}$. We will denote by $\mathbf{G}_Y$ the ideal $(I_{\vs{\GIN(Y)}})^\sat$ in $S$ and call it the {\em double-generic initial ideal of $Y$}. 
\end{definition}

\begin{example}
Given a Hilbert polynomial $p(t)$ with Gotzmann number $r$, an integer $n$ and a term order $\prec$ on $S$, let $I$ be the ideal generated in $S$ by the greatest $\tbinom{n+r}{n}-p(r)$ terms of degree $r$ w.r.t.~$\prec$. If $p(t)$ is the Hilbert polynomial of $I$, then $I$ is a hilb-segment ideal (e.g.~\cite[Definition 3.7]{CLMR}). If the hilb-segment ideal exists (e.g.~\cite[Proposition 3.17]{CLMR}), it is the double-generic initial ideal $\mathbf{G}_Y$ of every irreducible component $Y$ of $\HilbScheme{p(t)}{n}$ containing it.
\end{example}

We end this section observing that, in addition to the irreducibile components, there are many other relevant subsets of the Hilbert scheme that are invariant under the action of \GLin. We now list a few examples which are obtained applying Proposition \ref{prop:componenti}.

\begin{example}\label{ex:singular}\ 

\begin{enumerate}[(i)]
\item The irreducible components of the singular locus of $\HilbScheme{p(t)}{n}$ are \GLin-stable. 
\item Let $f:\mathbb N\rightarrow \mathbb N$ be the Hilbert function of a subscheme of $\PP^n_K$ and let $W$ the locus of $\HilbScheme{p(t)}{n}$ of points corresponding to subschemes  $Z$ of $ \PP^n_K$   whose Hilbert function $H_Z$ behaves so that $H_V(t) \leq f(t)$ for every $t\in \mathbb N$. The irreducible components of $W$ are \GLin-stable   by Proposition \ref{prop:componenti}. Indeed, $W$ is stable under the action of $\mathrm{GL}$ and is closed by semicontinuity (for example, see \cite{mall}).
\item For any given integer $s$, let $W$ be the locus  of $\HilbScheme{p(t)}{n}$ of points corresponding to  subschemes of $\mathbb P^n_K$ whose regularity is lower than or equal to $s$, that is the Hilbert scheme with bounded regularity $\HilbSchemeR{p(t)}{n}{s}$ which is studied in \cite{BBR}. It is obviously stable under the action of $\mathrm{GL}$, and it is open by semicontinuity. Then, the irreducible components of the complementary $\HilbScheme{p(t)}{n}\setminus W$ (i.e.~the set of points of $\HilbScheme{p(t)}{n}$ corresponding to subschemes with regularity $\geq s+1$) are \GLin-stable.
\end{enumerate}
\end{example}

\begin{example}\label{ex:gore e punti}
 Let $p(t)=c$ be a constant Hilbert polynomial. 
\begin{enumerate}[(i)] 
\item The locus of $\HilbScheme{c}{n}$ of points corresponding to  schemes that are supported on a unique point is closed and stable under the action of $\mathrm{GL}$.  Thus, its irreducible components are \GLin-stable. 
\item For $p(t)=c$, the locus  of $\HilbScheme{c}{n}$ of points corresponding  to  Gorenstein schemes is an open subset,  stable under the action of $\mathrm{GL}$. Thus, the irreducible components of its closure are \GLin-stable. 
\item For $p(t)=c$,  the irreducible components of the closure of the locus  of $\HilbScheme{c}{n}$  of the Gorenstein schemes that are supported on a unique point are \GLin-stable. 
\item The irreducible components of the  locus of $\HilbScheme{c}{n}$ of points corresponding to non-reduced subschemes of $\PP^n_K$ are  \GLin-stable. 
\end{enumerate}
Many other examples can be obtained considering every locus of $\HilbScheme{p(t)}{n}$ that is defined as a subscheme of $\mathbb P^n_K$ by any property of its points, because such a locus is invariant under the action of $\mathrm{GL}$.
\end{example}


\section{\texorpdfstring{The partial order {$\prec\!\!\prec$} on the terms of $\wedge^q S_m$}{The partial order << on lists of terms on $\wedge^q S_m$}}
\label{sec:minoreminore}

In this section, we introduce a partial order $\prec\!\!\prec$ between finite subsets with the same cardinality $q$ of a totally ordered set $T$ and prove some properties.  Then, we will apply these results to the case of lists of terms in $S_m$ and extend them to terms in $\wedge^q S_m$. We obtain that the double-generic initial ideal of a \GLin-stable subset $Y$ of a Hilbert scheme is the maximum among all the Borel terms in $Y$ with respect to the partial order $\prec\!\!\prec$. 
This characterization gives a simple and deterministic method to recognize the double-generic initial ideal of $Y$ among all the Borel ideals in $Y$.

\begin{definition}\label{partial order}
Let $(T,\prec)$ be a totally  ordered set and consider $A, B\subseteq T$ containing $q$ distinct elements each. We write  $A\prec\!\!\prec B$ if there is a bijection $\omega:A\rightarrow B$ such that $a\preceq \omega(a)$, for every $a\in A$.
\end{definition}

It is quite obvious that $\prec\!\!\prec$ is a partial order and that in particular $A\prec\!\!\prec A$. The following technical result allows a better understanding of its meaning. 

\begin{proposition}\label{intrinseco} Let $(T,\prec)$ be a finite, ordered set and  $A, B$ be two subsets of $T$ containing $q$ distinct elements each. Further, we index the elements of $A=\lbrace a_1,\dots,a_q\rbrace$ so that $a_i\succeq a_{i+1}$ for every $i \in \lbrace 1,\dots,q-1\rbrace$, and similarly for $B$. 
The followings are equivalent:
\begin{enumerate}[(i)]
\item \label{intrinseco_i}$A\prec\!\!\prec B$; 
\item \label{intrinseco_ii}for every element $c \in T$, $\vert \{{a_i} : {a_i} \succeq c \}\vert \leq \vert\{b_j : b_j \succeq c \}\vert$;
\item \label{intrinseco_iii} for every $i \in \lbrace 1,\dots,q \rbrace$, ${a_i}\preceq {b_i}$;
\item \label{intrinseco_iv}$ \lbrace a_1,\dots,a_q \rbrace\setminus \lbrace b_1,\dots,b_q\rbrace  \prec\!\!\prec \lbrace b_1,\dots,b_q\rbrace \setminus \lbrace a_1,\dots,a_q \rbrace.$
\end{enumerate}
\end{proposition}

\begin{proof}  
We first prove that item (\ref{intrinseco_i}) implies item (\ref{intrinseco_ii}).  Observe that we can assume $c\in A\cup B$, by replacing $c$ if necessary by the smallest term w.r.t.~$\prec$ in $A\cup B$ which is greater than $c$.

If $c=a_s\in A$, then there are exactly $s$ elements in $A$ bigger than or equal to $c$: more precisely, $a_i\preceq c$ for $i\in\{1,\dots,s\}$. 
Consider the bijection $\omega: A\rightarrow B$ such that $\omega(a_i)\succeq a_i$, for every $i\in \lbrace 1,\dots, q\rbrace$. Since we have 
$\{b_i\in B\vert b_i\succeq a_s=c\}\supseteq \left\{\omega(a_j)\right\}_{ j\in \{1,\dots,s\}}$
then it is immediate that $\vert\{b_i\in B\vert b_i\succeq a_s=c\}\vert\geq s$.

Otherwise, if $c=b_s \in B$, then $\vert \{b_j \in B\vert b_j  \succeq b_s=c \}\vert=s$, and  for every $j>s$ we have $b_s \succ b_j\succeq \omega^{-1}(b_j)$. Hence $\vert \{a_j \in A\vert a_j \succeq b_s=c \}\vert \leq  s $ since it is contained in $A\setminus \{  \omega^{-1} (b_{s+1}), \dots , \omega^{-1}(b_q) \}$. 

Item (\ref{intrinseco_ii}) implies (\ref{intrinseco_iii}), by contradiction: if there is $j\in \lbrace 1,\dots, q\rbrace$ such that $a_j\succ b_j$, then $\vert \lbrace a_i\in A\vert a_i\succeq a_j \rbrace \vert=j>\vert \lbrace b_i \in B\vert b_i\succeq a_j \rbrace \vert$.

Finally, if item\eqref{intrinseco_iii} holds, then we can consider the bijection $\omega: A\rightarrow B$ defined as $\omega(a_i)=b_i$, which fulfills the Definition of $\prec\!\!\prec$.

The equivalence between item \eqref{intrinseco_ii} and \eqref{intrinseco_iv} is immediate.
\end{proof}

Proposition \ref{intrinseco}\eqref{intrinseco_iii} points out a \lq\lq natural\rq\rq\ bijection between the sets $A$ and $B$ which fulfills Definition  \ref{partial order}, but it does not mean that there are no other such bijections. If, for instance, $b_1 \succeq \dots  \succeq b_q\succ a_1\succeq \dots \succeq a_q$, then every bijection from $A$ to $B$ fulfills Definition \ref{partial order}.

If $\prec$ is a term order on $S$, then for every integer $m$ the couple $(S_m\cap\mathbb T,\prec)$ is a finite  ordered set. From now, we identify every normal expression $\tau_1\wedge \dots \wedge \tau_q\in \wedge^q S_m$ with the set $\lbrace \tau_1,\dots,\tau_q\rbrace\subset S_m\cap\mathbb T$. 

\begin{definition}\label{precprec}
For every two terms $\tau_1\wedge \dots \wedge \tau_q$ and $\sigma_1\wedge \dots \wedge \sigma_q\in \mathbf T_{S_m}^q$, we write
$\tau_1\wedge \dots \wedge \tau_q\prec\!\!\prec \sigma_1\wedge \dots \wedge \sigma_q$ if and only if
$\lbrace \tau_1,\dots,\tau_q\rbrace\prec\!\!\prec \lbrace \sigma_1,\dots,\sigma_q\rbrace$, according to Definition \ref{partial order}.
\end{definition}

Now, we can apply the partial order $\prec\!\!\prec$ of Definition \ref{precprec} and the results of Proposition \ref{intrinseco} to the terms of $\wedge^q S_m$. If $N$ is a set of terms of $\wedge^q S_m$, by $\max_{\prec\!\!\prec} N$ we denote  (if it exists) the maximum of $N$ w.r.t.~the order $\prec\!\!\prec$.

In Remark \ref{rem:ginpiugrande}, we observed that, for every point $V$ of $\GrassScheme{S_m}{q}$, we have $\IN(V)\preceq\GIN(V)$. Now, we observe there is a stronger relation between $\IN(V)$ and $\GIN(V)$. More generally, we prove that we can replace the order of \eqref{ordEisenbud} by the order $\prec\!\!\prec$ of Definition \ref{precprec} in the study of initial and generic initial extensors of irreducible closed subsets of $\GrassScheme{S_m}{q}$. 

\begin{theorem}\label{th:minoreminore}
Let $W$ be a subset of $\GrassScheme{S_m}{q}$ such that $\overline W$ is irreducible. Then
\begin{enumerate}[(i)]
\item \label{th:minoreminore_i} $\IN(W)=\max_{\prec\!\!\prec} \supp(W)$;
\item \label{th:minoreminore_ii} $\GIN(W)=\max_{\prec\!\!\prec} \supp(\O(W))$; 
\item \label{th:minoreminore_iii} $\IN(W)\prec\!\!\prec \GIN(W)$;
\item \label{th:minoreminore_iv} $\GIN(W)=\max_{\prec\!\!\prec}(\mathbf B_{S_m}^q \cap  \overline W)$.
\end{enumerate} 
\end{theorem}

\begin{proof} 
For what concerns statement \eqref{th:minoreminore_i}, thanks to Proposition \ref{prop:chiudere}\eqref{prop:chiudere_i} and Proposition \ref{prop:chiusura}\eqref{prop:chiusura_i}, we can assume that $W$ is closed.   Furthermore, if $V$ is the generic point of $W$, then $\IN(V)=\IN(W)$ and $\supp(V)=\supp(W)$ by Propositions \ref{prop:chiudere}\eqref{prop:chiudere_iii} and \ref{prop:chiusura}\eqref{prop:chiusura_iii}. Up to an extension of the field of scalars, it is sufficient to prove statement \eqref{th:minoreminore_i} for the $K$-point $V$.

Let $\IN(V)=\tau_1\wedge\dots\wedge \tau_q$.
We show that $\sigma_1\wedge\dots\wedge \sigma_q\prec\!\!\prec \tau_1\wedge\dots\wedge \tau_q$, for every $\sigma_1\wedge\dots\wedge \sigma_q\in \supp(V)\setminus\lbrace \IN(V)\rbrace$. 

We can choose for the $K$-vector space   $V$  the unique basis $f_1,\dots,f_q\in S_m$, with  $\In(f_i)=\tau_i$ and $f_i=\tau_i+\sum_{\eta_{j_i}\prec \tau_i} c_{i{j_i}} \eta_{j_i}=\sum_{\eta_{j_i}\preceq \tau_i} c_{i{j_i}} \eta_{j_i}$, where the coefficient of $\tau_i$ in the last summation is $1$. Consider $f_1\wedge \dots \wedge f_q=\sum_{\sigma_1\wedge\dots\wedge\sigma_q\in \supp(V)} \Delta_L(V) \sigma_1\wedge\dots \wedge \sigma_q$.

For every normal expression $\sigma_1\wedge\dots\wedge \sigma_q\in \supp(V)\setminus\lbrace \IN(V)\rbrace$, by construction of $f_1\wedge \dots \wedge f_q$ we have
$$\sigma_1\wedge\dots \wedge \sigma_q=sgn(\gamma) \eta_{j_{\gamma(1)}}\wedge \dots \wedge \eta_{j_{\gamma(q)}},$$ 
for some $\gamma$ permutation of $\lbrace 1,\dots,q\rbrace$ and $\eta_{j_{\gamma(\ell)}}$ appearing with non-null coefficient in $f_{\gamma(\ell)}$. 

Hence, we can consider the bijection $\omega:\lbrace\sigma_1,\dots,\sigma_q \rbrace\rightarrow \lbrace \tau_1,\dots,\tau_q\rbrace$ such that $\omega(\sigma_\ell)=\tau_{\gamma(\ell)}$. This bijection $\omega$ fulfills Definition \ref{partial order}.

To prove \eqref{th:minoreminore_ii}, we observe that 
\[\GIN(W)=\IN(\mathcal O(W))= \IN(\overline{\mathcal O(W)})= \max_{\prec\!\!\prec} \supp(\overline{\O(W)}),\]
 by definition of generic initial extensor, by Proposition \ref{prop:chiusura} and by fact \eqref{th:minoreminore_i}. Then, we conclude because $\supp(\overline{\O(W)})= \supp(\O(W))$ by Proposition \ref{prop:chiudere}. 

To prove \eqref{th:minoreminore_iii}, now it is enough to observe that $\supp(W)$ is included in $\supp(\O(W))$. 

Fact \eqref{th:minoreminore_iv} follows by the previous items and by Theorem \ref{th:eisenbud}\eqref{th:eisenbud_iii}.
\end{proof}

\begin{remark}
It is important to observe that {\em the statement of Theorem \ref{th:minoreminore} does not hold true in the weaker hypothesis that the subset $W$ is only closed}. Indeed, the generic initial extensors of its irreducible components do not need to be comparable by the partial order $\prec\!\!\prec$. This is a crucial point for the application of this result in subsection \ref{subsec:1}.
\end{remark}


Let $I$ and $J$ be two saturated monomial ideals in $S$ and assume that their Hilbert functions are equal in degree $m$, that is $\dim_K (I_m)=\dim_K (J_m) =q$.  If $I_m=\langle \sigma_1,\dots,\sigma_q\rangle $ and $J_m=\langle \tau_1,\dots,\tau_q\rangle$, then we can compare the sets of terms $\{ \sigma_1,\dots,\sigma_q\}, \{\tau_1,\dots,\tau_q\}$ w.r.t.~$\prec\!\!\prec $, and if $\{ \sigma_1,\dots,\sigma_q\} \prec\!\!\prec \{ \tau_,\dots,\tau_q\}$, by abuse of notation we will simply write $I_m\prec\!\!\prec J_m$.

This relation between the degree $m$ components of the two ideals $I$ and $J$ cannot be considerd as a relation between the two ideals. 
Indeed, if $m'\neq m$ then the components of degree $m'$ of $I$ and $J$ may have different dimensions, hence they are no more comparable using $\prec\!\!\prec$. 

We now present some interesting 
cases, where the relation $\prec\!\!\prec$ among the degree $m$ parts of two monomial ideals lying on the same Hilbert scheme is preserved when passing to another degree. The first one concerns the double-generic initial ideal of a \GLin-stable subset and follows from Theorem \ref{th:minoreminore}.

\begin{corollary}\label{ginmaggiore} Let $Y$ be any \GLin-stable subset of $ \HilbScheme{p(t)}{n}$ and $r$ be the Gotzmann number of the Hilbert polynomial $p(t)$. If $J$ is the saturated ideal defining a point of $Y$, then for every $m\geq r$
\begin{equation}\label{minoreminoreOK}
J_m \prec\!\!\prec (\mathbf{G}_Y)_m  .
\end{equation}
\end{corollary}

\begin{proof}
Taking the embedding of $Y$ in $\GrassScheme{S_m}{q}$, the point 
defined by $J$ is the $K$-vector space $J_m$. Then, $\wedge^q J_m=\IN(J_m)$  belongs to  $\supp(Y)$. By Theorem \ref{th:minoreminore}, we have $\GIN(Y)=\IN(Y)=\max_{\prec\!\!\prec} \supp(Y)$; hence in particular $\wedge^q J_m  \ {\prec\!\!\prec} \ \GIN(Y)$. 
We can conclude because $\GIN(J_m)$ is the homogeneous component of degree $m$ of the double-generic initial ideal \hskip -0.5mm $\mathbf{G}_Y$ of~$Y$, by Corollary~\ref{prop:GIGI}.
\end{proof}

\begin{remark} Let $Y$ be  \GLin-stable.   If $r_Y$ is the maximum among the regularities of the points of $Y$, \eqref{minoreminoreOK}  holds true also for every integer $m'$,  $ r_Y\leq m' <r$, by \cite[Theorem 1.2]{BBR}.  
\end{remark}

Let $I$ and $J$ be any two saturated monomial ideals defining  points of the same Hilbert scheme. Now, we show that $J_m \prec\!\!\prec I_m$ is equivalent to $J_{m+1} \prec\!\!\prec I_{m+1}$  if $J$ and $I$ are Borel-fixed and $\prec$ is the degrevlex term order. Moreover, if $p(t)$ \hskip -0.5mm is a constant polynomial, this result holds true for every term order on~$S$.
For every term $\tau = x_0^{\alpha_0}\cdot\dots\cdot x_n^{\alpha_n}$, we set $\min(\tau):=\min\{i\in \{0,\dots,n\} : \alpha_i\not= 0\}$.

We recall that if the monomial ideal $J$ is Borel-fixed, and $J=J_{\geq  m}$, with $m\geq \reg(J^\sat)$, then $J$ is  a stable ideal (see \cite{Seiler2009I,Seiler2009II} and the references therein for details about stable ideals and their properties).
We extend \cite[Definition 2.7]{mall} to such an ideal: we call  {\em growth-vector of $J_m$} the vector $gv(J_m):=(v_0,\dots,v_n)$, with $v_i:=\vert\{ \tau \in J_m\cap \mathbb T_m : \min(\tau)=i \}\vert$.

\begin{lemma}\label{lemma:growth-vector}
Let $J$ be any $m$-regular Borel-fixed ideal with Hilbert polynomial $p(t)$. Then the growth-vector of $J_m$ depends only on $p(t)$ and $m$.
\end{lemma}

\begin{proof}
Let $v=(v_0,\dots,v_n)$ be the growth-vector of $J_m$. By \cite[Lemma 1.1]{EK}, for every $t\geq m$, we have
\begin{equation}
q(t)=\dim S_t -p(t)=\sum _{i=0}^n v_i \tbinom{t-m+i}{i}.
\end{equation}
Since $J$ is Borel-fixed and $m$-regular, then $J_{\geq m}$ is stable and the term $x_n^m$ belongs to $J_m$ \cite[Proposition 4.4]{Seiler2009II}, and we obtain that $v_n=1$. By induction $v_j$ is the product of $j!$ times the leading coefficient of $q(t)-\sum _{i=j+1}^{n} v_i\binom{t-m+i}{i}$.
\end{proof}

From now on, let $J$ and $I$ be $m$-regular {Borel-fixed} ideals in $S$ such that $\Proj(S/J)$ and $\Proj(S/I)$ share the same Hilbert polynomial $p(t)$. 

Given a term $\tau=x_0^{\alpha_0}\dots x_n^{\alpha_n} \in \mathbb T$, in the following we let $\partial_{x_i}(\tau):=\alpha_i$.

\begin{lemma}\label{minima}
Let $\prec$ be the degrevlex term order on $S$. Assume $J_m \prec\!\!\prec I_m$, and let $\omega\colon  J\cap \mathbb{T}_m\rightarrow I\cap \mathbb{T}_m$ be any function such that $\tau \preceq \omega(\tau)$ for every $\tau \in J_m\cap \mathbb T$.
If $\ell:=\min(\tau)$ then $\ell=\min(\omega(\tau))$ and $\partial_{x_\ell} (\tau)\geq\partial_{x_\ell} (\omega(\tau))$.
\end{lemma}

\begin{proof}
Since we are using the degrevlex term order, then $\min(\tau)\leq \min(\omega(\tau))$  for every $\tau \in J_m\cap \mathbb T$. Hence, $\min(\tau)=\min(\omega(\tau))$, because the growth vector of $J_m$ is the same as the one of $I_m$ by Lemma \ref{lemma:growth-vector}. 
The second part of the statement follows directly from the definition of the degrevlex term order.
\end{proof}

\begin{proposition}\label{prop: m cresce} 
If $\prec$ is the degrevlex term order on $S$, then $J_m \prec\!\!\prec I_m$ if and only if $J_{m+1} \prec\!\!\prec I_{m+1}$.
\end{proposition}

\begin{proof}
First, assume that $J_m \prec\!\!\prec I_m$ and let $\omega_m\colon J\cap \mathbb{T}_{m}\rightarrow I\cap \mathbb{T}_{m}$ be a bijective function such that $\tau \preceq \omega_m(\tau)$. 

Every term in $J_{m+1}$ is of kind $\tau x_\ell$ for a unique $\tau \in J_m$ and $\ell\leq \min(\tau)$, and the same is for $I_{m+1}$ \cite[Lemma 1.1]{EK}. 
For every $\tau x_l\in J\cap \mathbb{T}_{m+1}$, we define  $\omega_{m+1}(\tau x_l):=\omega_m(\tau)x_l$. Since $\prec$ is a term order, it is immediate that $\omega_{m+1}(\tau x_l)\succ \tau x_l$.
Then, by Lemma \ref{minima} we see that the function $\omega_{m+1}\colon J\cap \mathbb{T}_{m+1}\rightarrow I\cap \mathbb{T}_{m+1}$ is bijective.

Vice versa, assume that $J_{m+1} \prec\!\!\prec I_{m+1}$ and let $\omega_{m+1}\colon  J\cap \mathbb{T}_{m+1}\rightarrow I\cap \mathbb{T}_{m+1}$ be a bijective function such that $\tau x_\ell\preceq \omega_{m+1}(\tau x_\ell )$. We now construct $\omega_m\colon J\cap \mathbb{T}_{m}\rightarrow I\cap \mathbb{T}_{m}$. 

We observe that there is a bijection between the terms in $J_m$ and the subset  $J_{m+1}'$ of terms in $J_{m+1}$ that are divisible by the square of their minimal variable; indeed we obtain such a bijection  associating to $\tau \in J_m\cap\mathbb T$ the term $\tau \cdot x_{\min(\tau)}\in J_{m+1}$.  The same happens for the subset $I_{m+1}'\subset I_{m+1}$ defined in the same way  as $J'_{m+1}$. 

By Lemma \ref{minima}, we obtain $\omega_{m+1}^{-1}(I_{m+1}') \subseteq J_{m+1}' $. On the other hand $I_{m+1}'$ and $J_{m+1}'$ have the same cardinality (that of $J\cap \mathbb{T}_{m}$ and $I\cap \mathbb{T}_{m}$). Hence, $\omega_{m+1}^{-1}(I_{m+1}') =  J_{m+1}'$ and we obtain the bijection $\omega_m$ by setting $\omega_m(\tau)=\omega_{m+1}(\tau\cdot x_\ell) /x_\ell$ where $\ell:=\min(\tau)$.
\end{proof}

\begin{proposition}\label{m cresce con polinomio costante}
If $p(t)=d$ is a constant Hilbert polynomial, then $J_m \prec\!\!\prec I_m$ if and only if $J_{m+1} \prec\!\!\prec  I_{m+1}$, for every term order $\prec$ on $S$.
\end{proposition}

\begin{proof}
Being the Hilbert polynomial constant, $J_t$ and $I_t$ contain all the terms of degree $t$ in the variables $x_1,\dots,x_n$, for every $t\geq m$. Hence, we can consider only the terms $\tau$ with $\min(\tau)=0$ and conclude by Proposition \ref{intrinseco}\eqref{intrinseco_iv} and Lemma \ref{lemma:growth-vector}. 
\end{proof}


\section{Applications}
\label{sec:applications}

We always consider $\HilbScheme{p(t)}{n}$ embedded in $\GrassScheme{S_{m}}{q}$ for some $m\geq r$, and, for the sake of semplicity, we denote by 
$\mathbf B_{S_m}^q \cap \HilbScheme{p(t)}{n}$ the set of Borel-fixed extensor terms corresponding to points of $\HilbScheme{p(t)}{n}$. Recall that {\em all} the terms of $\mathbf B_{S_m}^q\cap \HilbScheme{p(t)}{n}$ can be obtained by the algorithms presented in \cite{CLMR,PL} in  characteristic 0, and in \cite{B2014} for every characteristic.

In this section, we show how the properties of the generic initial ideal and of the partial term order ${\prec\!\!\prec}$ in $\mathbf B_{S_m}^q \cap \HilbScheme{p(t)}{n}$ can be used to investigate the topological structure and the rationality of the irreducible components of a Hilbert scheme. 
The following first result singles out a condition that every Borel-fixed ideal defining a point of a \GLin-stable subset must satisfy.


\begin{proposition} \label{cor:condizione necessaria}
 Let $Y$ be a $\GLin$-stable subset of $\HilbScheme{p(t)}{n}$.
\begin{enumerate}[(i)]
\item If $L\in\mathbf B_{S_m}^q \cap \HilbScheme{p(t)}{n}$ and $\vs{L} {\prec\!\!\not\prec } (\mathbf{G}_Y)_m$,  then $\vs{L} \notin Y$.
\item If $V$ is a $K$-point of $\HilbScheme{p(t)}{n}$ and $\vs{\GIN(V)} {\prec\!\!\not\prec }  (\mathbf{G}_Y)_m$,  then $V \notin Y$.
\end{enumerate}
\end{proposition}

\begin{proof}
The statement is an immediate consequence of Theorem \ref{th:minoreminore} and of Corollary \ref{ginmaggiore}. Indeed, by these results, the $m$-degree homogeneous part of the double-generic initial ideal of $Y$ determines the maximal term among the Borel terms in $Y$ and, hence, among all the terms in $Y$.
\end{proof}

From Proposition \ref{cor:condizione necessaria}, we have therefore that if $\vs{L}$ belongs to a \GLin-stable subset $Y$, where $L$ is a term, then necessarily $\vs{L} {\prec\!\!\prec } (\mathbf{G}_Y)_m$.

\subsection{\texorpdfstring{Detection of different components in a Hilbert scheme}{Detection of different components in a Hilbert scheme}}
\label{subsec:1}

In this subsection, we see that some interesting lower bounds for the number of irreducible components of a Hilbert scheme spring out from the properties of the double-generic initial ideal and of the partial term order ${\prec\!\!\prec}$.

\begin{proposition}\label{prop:numero componenti}
Let $\prec$ be a term order in $S$ and $M_{\prec\!\!\prec}$ be the number of the maximal terms in $\mathbf B_{S_m}^q \cap \HilbScheme{p(t)}{n}$ w.r.t.~$\prec\!\!\prec$. Then, there are at least $M_{\prec\!\!\prec}$ irreducible components in $\HilbScheme{p(t)}{n}$.
\end{proposition}

\begin{proof}
The statement follows directly by Propositions \ref{prop:componenti}\eqref{prop:componenti_0} and \ref{cor:condizione necessaria} and Theorem~\ref{th:minoreminore}.
\end{proof}

Assuming $n>2$, if $J\subset S$ is a Borel-fixed ideal, we denote by $\sat_{x_0,x_1}(J)$ the ideal generated by the evaluations at $(1,1,x_2,\dots,x_n)$ of the terms generators of $J$ and call it the {\em $x_0,x_1$-saturation of $J$} (see \cite{R1}).
Denote by $\Lambda$ the term in $\mathbf B_{S_m}^q\cap \HilbScheme{p(t)}{n}$ corresponding to the unique saturated lex-segment ideal in $S$ whose Hilbert polynomial is $p(t)$.

\begin{corollary}\label{cor:numero componenti}
If $\mathrm{char}(K)=0$ and $\sat_{x_0,x_1}(\vs{L})\not= \sat_{x_0,x_1}(\vs{\Lambda})$ for every maximal term $L \in\mathbf B_{S_m}^q \cap \HilbScheme{p(t)}{n} \setminus \{\Lambda\}$, then there are at least $M_{\prec\!\!\prec}+1$ irreducible components in $\HilbScheme{p(t)}{n}$.
\end{corollary}

\begin{proof}
It is enough to apply Proposition \ref{prop:numero componenti} and \cite[Theorem 6]{R1}.
\end{proof}

\begin{remark}
If $L$ is a maximal term in $\mathbf B_{S_m}^q\cap \HilbScheme{p(t)}{n}$, then the corresponding ideal $(\vs{L})$ is strongly stable, also if the field $K$ has positive characteristic. Indeed, let $L'=\sigma_1\wedge \dots \wedge  \sigma_q$ be a Borel-fixed extensor term whose corresponding ideal is not strongly stable. Then, over any   field of charactestic zero, $\tau_1\wedge \dots \wedge  \tau_q:=\GIN(\langle \sigma_1,  \dots, \sigma_q\rangle )$ is a Borel term corresponding to a strongly stable ideal with Hilbert polynomial $p(t)$ and $\tau_1\wedge \dots \wedge  \tau_q\succ\!\!\succ \sigma_1\wedge \dots \wedge  \sigma_q$.   Therefore,  in $\mathbf B_{S_m}^q\cap \HilbScheme{p(t)}{n}$ there is at least a term which is  $\succ\!\!\succ$ of $L'$.   Following the terminology introduced in \cite{CS2013},  the  ideal $(\tau_1, \dots , \tau_q) $ is the {\it zero-generic initial ideal of} $(\sigma_1,\dots ,  \sigma_q)$.
\end{remark}

The bounds of Proposition \ref{prop:numero componenti} and of Corollary \ref{cor:numero componenti} are not meaningful in two cases. The first is when $p(t)$ is a constant Hilbert polynomial, because then all the Borel-fixed ideals are on the same component, hence for every term order on $S$ we will find a unique maximal term w.r.t.~$\prec\!\!\prec$. The second case is when the given term order $\preceq$ on $S$ is the deglex one, because there is the unique maximal term corresponding to the lex-segment ideal. Anyway, in general we can get useful information, although the lower bound depends on the term order given on $S$, as the following example shows. 

\begin{example}\label{ex:4 componenti}
For $n=3$ and $p(t)=7t-5$, the Gotzmann number is $r=16$. We get the complete list of the $112$ strongly stable ideals in $\HilbScheme{7t-5}{3}$ by \cite{CLMR,PL} and compare their intersections with $S_{16}$ w.r.t.~$\prec\!\!\prec$ for several  term orders in $S$ . As just observed,  there is only one maximal element for the lexicographic  term order. If we consider the term order on $S$ given by the weight vector $[w_0=1,w_1=2,w_2=9,w_3=12]$ (and ties broken by lex) we obtain 
two maximal terms  corresponding to the ideals with saturations $\mathfrak b_1:=(x_3^3,x_3^2x_2,$ $x_3x_2^2,x_3^2x_1,x_2^5)$ and $\mathfrak b_2:=(x_3^2,x_3x_2^3,x_2^4)$, respectively. Computing the $x_0,x_1$-saturation, we see that neither of them  lies on the component containing the lex-segment ideal, because $(\vs{\Lambda})^{\sat}=(x_3,x_2^8,x_2^7x_1^9)$. Thus, there are at least $3$ irreducible components in $\HilbScheme{p(t)}{n}$  by Corollary \ref{cor:numero componenti}. 

If we choose the degrevlex term order, we find $4$ maximal terms corresponding to the ideals with the following saturations: the ideals $\mathfrak b_1$, $\mathfrak b_2$ previously considered and
\begin{equation*}
\begin{split}
&\mathfrak b_3:=(x_3^3,x_3^2x_2^2,x_3x_2^3,x_3^2x_2x_1,x_3x_2^2x_1,x_3^2x_1^2, x_3x_2x_1^2,
x_3x_1^3,x_2^7),\\
&\mathfrak b_4:=(x_3^3,x_3^2x_2,x_3x_2^2,x_3^2x_1^2,x_3x_2x_1^2,x_2^6).
\end{split}
\end{equation*}
We conclude there are at least $4$ irreducible components in $\HilbScheme{p(t)}{n}$, by Proposition \ref{prop:numero componenti} since
in this case  the hypothesis of Corollary \ref{cor:numero componenti} does not hold.
\end{example}

\subsection{Maximal Hilbert function in a \GLin-stable subset}\label{subsec:2}
In this subsection, if $f$ and $g$ are two numerical functions, we say that {\em $f$ is greater than $g$} if $f(t)\geq g(t)$, for every $t\in\mathbb N$, and write $f\geq g$.

As we have already recalled in Section \ref{sec:minoreminore}, if a monomial ideal $J$ is Borel-fixed, and $J=(J_m)$, with $m\geq \reg(J^\sat)$, then $J$ is stable. 

\begin{theorem}\label{funzione massima}
Let $\prec$ be the degrevlex term order in $S$. If $V$ and $V'$ are two $K$-points of $\HilbScheme{p(t)}{n}$ such that $J:=I_V$ and $I:=I_{V'}$ are Borel-fixed ideals, then
$$V \prec\!\!\prec V' \ \Rightarrow \ \dim_K(J^{\sat})_t\geq \dim_K(I^{\sat})_t, \ \forall \ t\geq 0.$$
In particular, if $Y$ is a \GLin-stable subset  of $\HilbScheme{p(t)}{n}$, the Hilbert function of $\mathrm{Proj}(S/\mathbf{G}_Y)$ is the maximum among the Hilbert functions of $\mathrm{Proj}(S/H)$, where $H$ varies among the saturated ideals defining points of $Y$.
\end{theorem}

\begin{proof}
It is enough to prove $\dim_K(J^{\sat})_t\geq \dim_K(I^{\sat})_t$ for every $t<m$, because $V$ and $V'$ are points of the same Hilbert scheme and $m$ is an upper bound for the regularities of both $J$ and $I$.

For every $t<m$, $\dim_K(J^{\mathrm{sat}})_t$ is the number of terms of $J_m$ which are divisible by $x_0^{m-t}$, because $J$ is {Borel-fixed}.
Since $V'\succ \!\!\succ V$, we can apply Proposition \ref{intrinseco}  \eqref{intrinseco_ii}  to $c=x_n^{t}x_0^{m-t}$ and see that  the number of terms in $\vs{V}$ divisible  by $x_0^{m-t}$ is larger than or equal to those in $V'$. Hence, we obtain $\dim_K(J^{\mathrm{sat}})_t\geq \dim_K(I^{\mathrm{sat}})_t$.

The last statement follows from Theorem \ref{th:minoreminore}\eqref{th:minoreminore_ii} and the fact that, for every homogeneous polynomial ideal $H$, $\mathrm{gin}_{\prec}(H^{\sat})=\mathrm{gin}_{\prec}(H)^{\sat}$, being $\prec$ the degrevlex term order.
\end{proof}

\begin{remark}
By Theorems \ref{funzione massima} and \ref{th:minoreminore}\eqref{th:minoreminore_ii} we get another method to find different irreducible components of a Hilbert scheme that consists in detecting the maximal Hilbert functions of projective schemes with a given Hilbert polynomial. This method might be easier to use than the detection of the maximal Borel terms w.r.t.~the partial order $\prec\!\!\prec$. However, the detection of the maximal Hilbert functions gives a lower bound on the number of irreducible components which is far from being sharp
: for instance, in Example \ref{ex:4 componenti} we find $4$ maximal Borel-fixed terms but there are only $2$ maximal Hilbert functions.
\end{remark}

\begin{remark} \label{rem:maximum}
The existence of the maximum among the Hilbert functions on a \GLin-stable subset of $\HilbScheme{p(t)}{n}$ can be proved by semicontinuity in the following way, although we observe that Theorem \ref{funzione massima} gives a constructive answer. Let $Y$ be a \GLin-stable subset of $\HilbScheme{p(t)}{n}$, $m$ be an upper bound on the Ca\-stel\-nuo\-vo-Mumford regularity of points in $Y$. We define the following numerical function $f:\mathbb N\rightarrow\mathbb N$
\[f(t)=\begin{cases}
\max\lbrace H_{S/I}(t)\ \vert\ I\in Y\rbrace, \quad \text {if } 1\leq t\leq m\\
p(t), \quad \text{otherwise.}
\end{cases}
\]
For every $1\leq s\leq m$, the subset $A_s=\lbrace I\in Y\vert H_{S/I}(s)\geq f(s)\rbrace$ of $Y$ is open by semicontinuity \cite[Remark 12.7.1 in chapter III]{H77}. Hence $\cap_{s=1}^m A_s$ is an open subset of $Y$ and it is non-empty, because every $A_s$ is non-empty by construction of $f$. Thus, there is an open subset of ideals $I\in Y$ having maximal Hilbert function $f$. 
\end{remark}
It would be nice to find a result analogous to that of Theorem \ref{funzione massima} for the deglex term order. We state the following conjecture.

\begin{conj}\label{conjLex}
Let $Y$ be a \GLin-stable subset of $\HilbScheme{p(t)}{n}$ and $\preceq$ be the deglex term order. Then, the Hilbert function of $\mathrm{Proj}(S/\mathbf{G}_Y)$ is the minimum among the Hilbert functions of $\mathrm{Proj}(S/H)$, where $H$ varies among the ideals defining a point of $Y$.
\end{conj}

\subsection{Rational components of a Hilbert scheme}\label{subsec:rat}
The following results deal with the rationality of irreducible components in a Hilbert scheme. The main tools we use are some of the features of double-generic initial ideals together with arguments introduced in \cite{FR,LR,RT}.

\begin{theorem}\label{razionale}
Let $Y$ be an isolated irreducible component of $\HilbScheme{p(t)}{n}$. If $\GIN(Y)$ corresponds to a smooth point in $Y$, then $Y$ is rational.
\end{theorem}

\begin{proof}
Let $L:=\GIN(Y)$ and recall that $\mathcal V_L$ is the set of points of  $\GrassScheme{S_m}{q}$ having $L$ as initial extensor (see \eqref{eq:V tondo con L}). By Proposition \ref{prop:componenti}\eqref{prop:componenti_i}, $L$ is equal to $\IN(Y)$ and, by Proposition \ref{lemma:UL}, $\mathcal  V_L \cap Y$ is a dense open subset of $Y$. 

By \cite[Lemma 3.2]{RT} and \cite{LR}, the set of ideals in $S$ having a given monomial ideal $J$ as initial ideal w.r.t.~any given term order 
can be endowed with a structure of homogeneous scheme $X$ w.r.t.~a non-standard grading. The reduced scheme structure $X^{red}$ and the isolated irreducible components of $X$ turn out to be homogeneous too \cite[Corollary 2.7]{FR}. Moreover, $X$ is connected, because every isolated irreducible component contains $J$, and every isolated irreducible component of $X$ that is smooth at $J$ is isomorphic to an affine space \cite[Corollary 3.3]{FR}. 

We now apply these results to the monomial ideal $J=(\vs{L})$. Note that it is enough to consider the reduced scheme structure $X^{red}$ because we deal with the isolated irreducible components of $X$ that are smooth at the point $J$. By definition of Hilbert scheme, $X^{red}$ and $(\mathcal V_L \cap \HilbScheme{p(t)}{n})^{red}$ are isomorphic. Moreover, by Proposition \ref{lemma:UL}\eqref{lemma:UL_iv}, the set $(\mathcal V_L \cap \HilbScheme{p(t)}{n})^{red}$ contains the open subset $Y':=\mathcal V_L\cap Y$ of $Y$ which is one of the irreducible components of $(\mathcal V_L \cap \HilbScheme{p(t)}{n})^{red}$ and, hence, preserves a homogeneous scheme structure. Being $\vs{L}$ a smooth point of $Y$, it is also smooth on $Y'$. Thus, $Y'$ is isomorphic to an affine space and $Y$ is rational.
\end{proof} 

As an application of Theorem \ref{razionale} we recover the well-known fact that the irreducible component of a Hilbert scheme containing the unique saturated lex-segment ideal is rational (e.g.~\cite{LR}). We now present a new result in the following corollary, where we focus on the 2-codimensional Cohen-Macaulay points  of a Hilbert scheme.

\begin{corollary}\label{prop:razionale aCM codim 2}
Let $p(t)$ be a Hilbert polynomial of degree $n-2$. Every irreducible component $Y$ of $\HilbScheme{p(t)}{n}$ containing a point $V$ corresponding to an arithmetically Cohen-Macaulay subscheme is rational.\\
Furthermore, if $L=\GIN(V) $, then $\mathcal V_L$ is isomorphic to an affine space and $\mathcal U_L$ is the Cohen-Macaulay locus of $Y$.
\end{corollary}

\begin{proof}
Recall that the subset $C$ of $\HilbScheme{p(t)}{n}$ formed by the points corresponding to arithmetically Cohen-Macaulay subschemes is an open subset containing $V$ \cite[Th\'{e}or\`{e}me (12.2.1)(vii)]{Gro}. Moreover, $V$ and every other point defining an arithmetically Cohen-Macaulay subscheme of codimension $2$  correspond to smooth points in $\HilbScheme{p(t)}{n}$ (see \cite{Fo}  for $n=2$ and \cite[Theorem 2(i)]{Elli} for $n\geq 3$).  Hence, $C\cap Y$ is a smooth,  non-empty open subset of $Y$.   It is well-known that if we choose the  degrevlex term order, then also $\Gin(I_V)^{\sat}=\Gin(I_V^{\sat})$ defines  an  arithmetically Cohen-Macaulay subscheme of $\PP^n_K$.  By Theorem \ref{gin appartiene}, the ideal $\Gin(I_V^{\sat})$ is the double-generic  initial ideal $\mathbf{G}_Y$ of $Y$, and its homogeneous component of degree $m$ is $\vs{L}$ with  $L=\GIN(Y)=\GIN(V)$.   
Thus, Theorem~\ref{razionale} allows us to conclude that $\mathcal V_L$ is an affine space and we observe that, by definition, $V$ belongs to $\mathcal U_L$. Note that the result we obtain holds for every $V\in C\cap Y$, hence $\mathcal U_L=C\cap Y$. 
\end{proof}

\begin{example} \label{ex:tommasino}
In this example we apply Theorem \ref{razionale} to the double-generic initial ideal $\mathbf{G}_Y$  of an irreducible component $Y$ of a Hilbert scheme, which is smooth for $Y$, but which is not smooth for the Hilbert scheme. 

Let us consider the Hilbert scheme $\HilbScheme{3t+2}{3}$ over a field of null  characteristic. There are $4$ saturated Borel ideals corresponding to points on this Hilbert scheme
$$\mathfrak b_1:=(x_{{3}},x_{{2}}^{4},x_{{1}}^{2}  x_{{2}}^{3} ), \quad \mathfrak  b_2:=(x_{{3}}^{2},x_{{2}}x_{{3}},x_{{1}}x_{{3}},x_{{2}}^{4},x_{{1}} x_2^3),$$
$$\mathfrak  b_3:=( x_{{3}}^{2},x_{{2}}x_{{3}},x_{{2}}^{3},x_{{1}}^{2}x_{{3}}), \quad\mathfrak  b_4:=( x_{{3}}^{2},x_{{2}}x_{{3}},x_{{2}}^{3},x_{{1}}x_{{2}}^{2}).$$
The ideal $\mathfrak b_1$ corresponds to the lex-segment point of $\HilbScheme{3t+2}{3}$. It is well-known that such a point belongs to a unique component, that we denote by $Y_1$, and this component is rational. By a direct computation, we find that  the dimension of $Y_1$ is $18$ and that its general point corresponds to the  union of a plane cubic curve and two isolated points. By \cite[Theorem 6]{R1}, we see that also $\mathfrak b_2$ and $\mathfrak b_3$ define points of $Y_1$, while $\mathfrak b_4$ does not.

If we choose the degrevlex term order, we find that $\mathfrak b_4$ is the maximum w.r.t.~${\prec\!\!\prec}$ of the Borel ideals. By a direct computation involving marked schemes (see \cite{CR11,BCLR,BLR}), using for instance either the Singular library \cite{michela,DGPS} or the algorithm described in \cite{BCR2}, we obtain that $\mathfrak b_4$ is contained in two irreducible components $Y_2$ and $Y_3$. Therefore, the  point corresponding to $\mathfrak b_4$ is not smooth for the Hilbert scheme. However,  it turns out to be smooth for both $Y_2$ and $Y_3$, and by Theorem \ref{razionale} we get that $Y_2$ and $Y_3$ are both rational. 

To complete the description, by a direct computation we find that  the dimension of $Y_2$ is $12$ and its general point corresponds to the disjoint union of a conic and a line.

Here are the generators of one of the ideals we obtain after a random specialization of the $12$ free parameters:
{\small{{\begin{enumerate}[]
\item $f_1=x_3^2 + 990x_0^2-x_2^2-3 x_1x_2-x_1x_3-2 x_1^2+67 x_0 x_3-23 x_0 x_2-68 x_0 x_1$
\item $f_2= x_2x_3 + 484x_0^2-x_2^2-2 x_1x_3+4 x_1^2+22 x_0x_3-88 x_0x_1$
\item $f_3=x_2^3 -6538x_0^3+3 x_0x_2^2+4 \, x_1^2x_3-x_2x_1^2+30 x_0x_1x_2+8 x_0x_1x_3+286 x_0x_1^2+6\,x_1^3-4x_0^2 x_3-711x_0^2 x_2-386 x_0^2x_1$
\item $f_4= x_1x_2^2 + 1913x_0^3-3x_0x_2^2+x_1^2x_3-3 x_0x_1^{2}x_2-3x_1x_2+2x_0x_1x_3+69 x_0 x_1^2+5x_1^3-x_0^2x_3+22x_0^2 x_2-815 x_0^2 x_1$
\end{enumerate}}}}
and the primary decomposition of this ideal 
{\small\begin{enumerate}[]
\item $\mathfrak p_{{1}}=(3x_1+x_3+23x_0,x_2-2x_1+22x_0)$
\item $\mathfrak p_{{2}}=(-2x_1+x_3+22x_0-x_2,7x_1^2-2x_1x_2+72x_0x_1-7x_0x_2+x_2^2-645x_0^2)$.
\end{enumerate}}

The dimension of $Y_3$ is $15$ and its general point corresponds to the union of a twisted cubic curve and a point. Here are the generators of one of the ideals we obtain after a random specialization of the $15$ free parameters:

{\small{{\begin{enumerate}[]
\item $f'_1 = x_3^2 + 37x_0^2-18x_1^2-x_2^2-3x_1x_3-3x_1x_2+2x_0x_3-6x_0x_2+15x_0x_1$
\item $f'_2= x_2x_3 + 31x_0^2-12x_1^2-x_2^2-2x_1x_3+2x_1x_2+x_0x_3+16x_0x_1$
\item $f'_3=x_2^3 -150x_0^3+86x_0x_1^2-24x_1^3-10x_1^2x_3 +37x_0x_1x_3+x_2x_1^2-7x_0x_1x_2-30x_0^2x_3-x_0^2x_2-11x_0^2x_1$
\item $f'_4=x_1x_2^2 -x_0^3-2x_0x_1^2+x_0x_2^2-2x_1^2x_3+3x_0x_1x_3-x_2x_1^2-x_0x_1x_2+5x_0^2x_3-3x_0^2x_1$
\end{enumerate}}}}
and the primary decomposition of this ideal 
{\small\begin{enumerate}[]
\item $\mathfrak p'_1=(x_0+x_1,x_2-8x_0,x_3-7x_0)$
\item $\mathfrak p'_2=(-x_1x_2-2\,x_1x_3-2\,x_0x_1-x_0^2+x_2^2+5\,x_0x_3, x_2x_3+30x_0^2-4\,x_1x_3+x_1x_2-12\,x_1^2+6\,x_0x_3+14\,x_0x_1,$
$x_3^2+36x_0^2-5\,x_1x_3-4\,x_1x_2-18\,x_1^2+7\,x_0x_3-6\,x_0x_2+13\,x_0x_1)$.
\end{enumerate}}

Finally, computing the marked schemes on $\mathfrak b_2$ and $\mathfrak b_3$, we check that $Y_1$, $Y_2$ and $Y_3$ are the only  irreducible components of $\HilbScheme{3t+2}{3}$.
\end{example}

\section*{Acknowledgments}
 
The authors are grateful to Anthony Iarrobino for his generous suggestions and inspiring comments on the topic and the writing of this paper.

\providecommand{\bysame}{\leavevmode\hbox to3em{\hrulefill}\thinspace}
\providecommand{\MR}{\relax\ifhmode\unskip\space\fi MR }
\providecommand{\MRhref}[2]{%
  \href{http://www.ams.org/mathscinet-getitem?mr=#1}{#2}
}
\providecommand{\href}[2]{#2}

\end{document}